\documentclass[12pt]{amsart}
\usepackage[T2A]{fontenc}
\usepackage[english]{babel}
\usepackage{amsaddr}
\usepackage{amsmath}
\usepackage{amssymb}
\usepackage{amsfonts}
\usepackage{epsfig}
\usepackage{srcltx}
\usepackage{subfigure}
\usepackage{cite}
\usepackage{float}
\usepackage[a4paper,  mag=1000, includefoot,  left=2cm, right=2cm, top=2cm, bottom=2cm, headsep=1cm, footskip=1cm ]{geometry}
\usepackage[unicode,colorlinks]{hyperref}
\hypersetup{colorlinks=true, citecolor=blue, linkcolor=blue}

\newtheorem{Th}{Theorem}
\newtheorem{Lem}{Lemma}

\newtheorem{Rem}{Remark}
\newtheorem{Cor}{Corollary}
\begin{document}
\thispagestyle{empty}

\title[Stability and bifurcation phenomena in asymptotically Hamiltonian systems]
{Stability and bifurcation phenomena in asymptotically Hamiltonian systems}
\author{Oskar A. Sultanov}

\address{
Institute of Mathematics, Ufa Federal Research Centre, Russian Academy of Sciences,
\newline
112, Chernyshevsky street, Ufa, Russia, 450008}
\email{oasultanov@gmail.com}


\maketitle
{\small

{\small
\begin{quote}
\noindent{\bf Abstract.}
The influence of time-dependent perturbations on an autonomous Hamiltonian system with an equilibrium of center type is considered.
It is assumed that the perturbations decay at infinity in time and vanish at the equilibrium of the unperturbed system. In this case the stability and the long-term behaviour of trajectories depend on nonlinear and non-autonomous terms of the equations. The paper investigates bifurcations associated with a change of Lyapunov stability of the equilibrium and the emergence of new attracting or repelling states in the perturbed asymptotically autonomous system. The dependence of bifurcations on the structure of decaying perturbations is discussed.

\medskip


\end{quote}
}

\section{Introduction}

The influence of perturbations on the stability of solutions is a classical problem in the qualitative theory of differential equations. For autonomous systems, the solution of such a problem is effectively covered by the theory of stability and bifurcations~\cite{GH83,GL94,Hans07}. This paper is devoted to non-autonomous perturbations such that the perturbed system is asymptotically autonomous. Asymptotically autonomous systems were first considered in~\cite{Markus56}, where the relations between the solutions of the complete system and the solutions of the corresponding limiting autonomous system were discussed. A special class of asymptotically autonomous systems on the plane and conditions that guarantee the stability and almost periodicity of solutions were investigated in~\cite{WongBurton65}. A more wide class of systems on the plane was considered in~\cite{Grimmer69}, where the almost periodic solutions were approximated by solutions of the corresponding limiting systems. Note that under some conditions, the solutions of a complete system have the same asymptotic behavior as the solutions of the limiting system (see, for example,~\cite{Theim92}). However, this is not true in general. Several examples of non-autonomous systems whose solutions behave completely differently than the solutions of the corresponding limiting systems were studied in~\cite{Theim94}.

Bifurcations in non-autonomous systems have recently been discussed in several papers. In particular, scalar differential equations with time-dependent coefficients were considered in~\cite{LRS02}, where bifurcations are associated with the change of a pullback stability and the appearance of new stable states. Similar equations were studied in~\cite{KS05}, where the bifurcation was understood as a change in the structure of the pullback attractor. Bifurcations as a change in the structure of the domain of attraction were discussed for asymptotically autonomous equations in~\cite{Rassmusen08}, where some conditions ensuring the transfer of bifurcations in limiting equations to complete equations were described. The elements of general theory for non-autonomous systems are contained in~\cite{CP10}, where some particular bounded solutions were considered as bifurcating objects and the bifurcation was understood as a branching of solutions.

The present paper considers a class of asymptotically Hamiltonian systems with the equilibrium and investigates the effects of decaying time-dependent perturbations on the stability and bifurcations of solutions. To the best of our knowledge, such problems have not been thoroughly studied.

The paper is organized as follows. In section~\ref{sec1}, the mathematical formulation of the problem is given and the class of non-autonomous perturbations is described. The proposed method of stability and bifurcation analysis is based on a change of variables associated with a Lyapunov function for a complete asymptotically autonomous system. The construction of this transformation is described in section~\ref{sec2}. Section~\ref{sec3} is devoted to bifurcations associated with a change of the stability of the equilibrium. Bifurcations associated with limit cycles are discussed in section~\ref{sec32}. The results of sections~\ref{sec3} and~\ref{sec32} are applied in section~\ref{sec4} for a description of bifurcations in the complete system under various restrictions on the perturbations. The paper concludes with a brief discussion of the results obtained.

\section{Problem statement}
\label{sec1}

Consider the system of two differential equations:
\begin{gather}
\label{FulSys}
\frac{dx}{dt}=\partial_y H(x,y,t), \quad \frac{dy}{dt}=-\partial_x H(x,y,t)+F(x,y,t), \quad t>0.
\end{gather}
It is assumed that the functions $H(x,y,t)$ and $F(x,y,t)$ are infinitely differentiable and for every compact $D\in\mathbb R^2$ $H(x,y,t)\to H_0(x,y)$ and $F(x,y,t)\to 0$ as $t\to\infty$ for all $(x,y)\in D$. The limiting autonomous system with the Hamiltonian $H_0(x,y)$ is assumed to have the isolated fixed point $(0,0)$ of center type. Without loss of generality, it is assumed that
\begin{gather}\label{H0}
    H_0(x,y)=\frac{x^2+y^2}{2}+\mathcal O(r^3), \quad r=\sqrt{x^2+y^2}\to 0.
\end{gather}
It is also assumed that the level lines $\{(x,y)\in\mathbb R^2: H_0(x,y)=E\}$ define a family of closed curves on the phase space $(x,y)$ parameterized by the parameter $E$ for all $E\in(0,E_0]$, $E_0={\hbox{\rm const}}$.

Non-autonomous perturbations of the limiting system are described by the functions with power-law asymptotics:
\begin{gather}
\label{HF}
H(x,y,t)-H_0(x,y)=\sum_{k=1}^\infty t^{-\frac kq} H_k(x,y), \quad F(x,y,t)=\sum_{k=1}^\infty t^{-\frac kq} F_k(x,y), \quad t\to\infty, \quad q\in \mathbb Z_+.
\end{gather}
It is assumed that the perturbations preserve the fixed point $(0,0)$:
\begin{gather*}
\partial_x H(0,0,t)\equiv 0, \quad \partial_y H(0,0,t)\equiv 0, \quad F(0,0,t)\equiv 0.
\end{gather*}
The structure of the perturbations can be more complicated, for example, asymptotic series \eqref{HF} can differ from power series, or the coefficients of asymptotics can explicitly depend on $t$. Such perturbations, however, are not considered in the present paper. Note that the Painlev\'{e} equations~\cite{IKNF}, autoresonance models~\cite{OSLK13,OS16} and  synchronization models~\cite{PRK02,LK14} are reduced to non-autonomous systems of the form \eqref{FulSys}.

Our goal is to describe possible asymptotic regimes in the perturbed system and to reveal the role of decaying perturbations in the corresponding bifurcations. Here, the bifurcations are associated with a change of Lyapunov stability of the equilibrium and the emergence of new attracting or repelling states.

Let us note that the decaying perturbations do affect the stability of the system. A simple example is given by the following equation:
\begin{gather*}
\frac{d^2x}{dt^2}+x=\gamma t^{-\kappa} \frac{dx}{dt}, \quad \gamma,\kappa\in\mathbb R, \quad \gamma\neq 0,\ \ \kappa>0.
\end{gather*}
This equation in the variables $x,y=\dot x$ has the form \eqref{FulSys} with $H(x,y,t)\equiv H_0(x,y)\equiv (x^2+y^2)/2$ and $F(x,y,t)\equiv \gamma t^{-\kappa} y$. It can easily be checked that the unperturbed autonomous equation ($\gamma=0$) has the following general solution: $x(t;a,\varphi)=a \cos(\varphi+t)$. The long-term asymptotics for a two-parameter family of solutions to the perturbed equation ($\gamma\neq 0$) is constructed with using WKB approximations~\cite{Vazov}:
\begin{align*}
&x(t;a,\varphi)=a \big[\cos(\varphi+t)+\mathcal O(t^{1-\kappa})\big], \quad \kappa>1;\\
&x(t;a,\varphi)=a t^{\gamma/2} \big[\cos(\varphi+t)+\mathcal O(t^{-1})\big], \quad \kappa=1;\\
&x(t;a,\varphi)=a \exp\Big(\frac{\gamma t^{1-\kappa}}{1-\kappa}\Big) \Big\{\cos\big(\varphi+t+\mathcal O(t^{1-2\kappa})+\mathcal O(\log t)\big)+\mathcal O(t^{-\kappa})\Big\}, \quad \kappa<1,
\end{align*}
where $a,\varphi\in\mathbb R$ are arbitrary parameters. It follows that the stability of the trivial solution $x(t)\equiv 0$ or the fixed point $(0,0)$ depends on the parameters $\gamma$ and $\kappa$. In particular, if $\kappa>1$,  the fixed point is marginally stable. In this case, the solutions of the non-autonomous equation have the same behaviour as the solutions of the limiting equation. The fixed point becomes attracting if $\gamma<0$ (polynomially stable when $\kappa=1$ and exponentially stable when $0<\kappa<1$), and loses stability if $\gamma>0$. In the general case, the long-term asymptotics for solutions are obtained not so easily, and the stability of the equilibrium depends on nonlinear terms of equations.
The examples of nonlinear equations are contained in section~\ref{sec4}.

\section{Change of variables}
\label{sec2}

The proposed method of study of asymptotic regimes in system \eqref{FulSys} is based on the construction of appropriate Lyapunov functions. Recently, it was noted in~\cite{OS18,OS19,OS20} that such functions are effective in the asymptotic analysis of solutions to nonlinear non-autonomous systems. See also~\cite{Hapaev} for application of the second Lyapunov method to asymptotic analysis of equations with a small parameter. Here, a Lyapunov function is used as a new dependent variable. In this section, the construction of such function and the change of variables are presented in a form suitable for further bifurcation analysis of  system \eqref{FulSys}.

First, consider the limiting system
\begin{gather}
    \label{LimSys}
        \frac{dx}{dt}=\partial_y H_0(x,y), \quad \frac{dy}{dt}=-\partial_x H_0(x,y).
\end{gather}
To each level line $\{(x,y)\in\mathbb R^2: H_0(x,y)=E\}$, $E\in (0, E_0]$ there correspond a periodic solution
$x_0(t,E)$, $y_0(t,E)$ of system \eqref{LimSys} with a period $T(E)=2\pi/\omega(E)$, where $\omega(E)\neq 0$ for all $E\in[0,E_0]$ and  $\omega(E)=1+\mathcal O(E)$ as $E\to 0$. The value $E=0$ corresponds to the fixed point $(0,0)$.

Define auxiliary $2\pi$-periodic functions $X(\varphi,E)=x_0(\varphi/\omega(E),E)$ and $Y(\varphi,E)=y_0(\varphi/\omega(E),E)$, satisfying the system:
\begin{gather*}
    \omega(E)\frac{\partial X}{\partial \varphi}=\partial_Y H_0(X,Y), \quad
    \omega(E)\frac{\partial Y}{\partial \varphi}=-\partial_X H_0(X,Y).
\end{gather*}
These functions are used for rewriting system \eqref{FulSys} in the action-angle variables $(E,\varphi)$:
\begin{gather}
\label{exch1}
    x(t)=X (\varphi(t),E(t)), \quad y(t)=Y(\varphi(t),E(t)).
\end{gather}
From the identity $H_0(X(\varphi,E),Y(\varphi,E))\equiv E$ it follows that
\begin{gather*}
    \begin{vmatrix}
        \partial_\varphi X & \partial_E X\\
        \partial_\varphi Y& \partial_E Y
    \end{vmatrix} = \frac{1}{\omega(E)}\neq 0.
\end{gather*}
The last inequality guarantees the reversibility of transformation \eqref{exch1} for all $E\in (0,E_0)$ and $\varphi\in \mathbb R$. It can easily be checked that in new variables  $(E,\varphi)$ system \eqref{FulSys} takes the form:
\begin{gather}
\label{FulSys2}
    \frac{dE}{dt}=f(E,\varphi,t), \quad  \frac{d\varphi}{dt}=\omega(E)+g(E,\varphi,t),
\end{gather}
where
\begin{eqnarray*}
    f(E,\varphi,t) & \equiv &
    -\omega(E) \Big(\partial_\varphi  H\big(X(\varphi,E),Y(\varphi,E),t\big) - F\big(X(\varphi,E),Y(\varphi,E),t\big)\partial_\varphi X (\varphi,E)\Big), \\
    g(E,\varphi,t) & \equiv & \omega(E)\Big(\partial_E  H\big(X(\varphi,E),Y(\varphi,E),t\big) -1- F\big(X(\varphi,E),Y(\varphi,E),t\big)\partial_E X(\varphi,E)\Big)
\end{eqnarray*}
are $2\pi$-periodic functions with respect to $\varphi$. Since $(0,0)$ is the equilibrium of system \eqref{FulSys}, we have $f(0,\varphi,t)\equiv 0$. From \eqref{HF} it follows that
\begin{gather*}
    f(E,\varphi,t)=\sum_{k=1}^\infty t^{-\frac kq} f_k(E,\varphi), \quad
    g(E,\varphi,t)=\sum_{k=1}^\infty t^{-\frac kq} g_k(E,\varphi), \quad t\to\infty,
\end{gather*}
where
\begin{gather}\label{fg}
    \begin{split}
        f_k(E,\varphi)&\equiv -\omega(E) \Big(\partial_\varphi H_k\big(X(\varphi,E),Y(\varphi,E)\big)- F_k\big(X(\varphi,E),Y(\varphi,E)\big)\partial_\varphi X(\varphi,E)  \Big),
        \\
        g_k(E,\varphi) &\equiv \omega(E) \Big(\partial_E H_k\big(X(\varphi,E),Y(\varphi,E)\big)- F_k\big(X(\varphi,E),Y(\varphi,E)\big)\partial_E X(\varphi,E) \Big).
    \end{split}
\end{gather}

To simplify the first equation in \eqref{FulSys2}, we consider the transformation of the variable $E$ in the form:
\begin{gather}\label{VF}
    V_N(E,\varphi,t)=E+\sum_{k=1}^N t^{-\frac kq} v_k(E,\varphi),
\end{gather}
where the coefficients $v_k(E,\varphi)$ are chosen in such a way that the right-hand side of the equation for the new variable $v(t)\equiv V_N(E(t),\varphi(t),t)$ does not depend on $\varphi$ at least in the first terms of the asymptotics:
\begin{gather}
    \label{VEq}
    \frac{dv}{dt}=\sum_{k=1}^N t^{-\frac kq} \Lambda_k(v) + R_{N+1}(v,\varphi,t), \quad R_{N+1}(v,\varphi,t)=\mathcal O(t^{-\frac{N+1}{q}}), \quad t\to\infty.
\end{gather}
Under the transformation $(E,\varphi)\mapsto (v,\varphi)$ the form of the second equation in \eqref{FulSys2} changes slightly:
\begin{gather}
    \label{PhiEq}
    \frac{d\varphi}{dt}=\omega(v)+G_N(v,\varphi,t), \quad G_N(v,\varphi,t)=\sum_{k=1}^\infty t^{-\frac kq}  g_{N,k}(v,\varphi), \quad t\to\infty.
\end{gather}
Here, each function $g_{N,k}(v,\varphi)$ is $2\pi$-periodic with respect to $\varphi$ and is expressed through $v_1,\dots, v_k$ and $g_1,\dots,g_k$. For example,
\begin{eqnarray*}
  g_{N,1}(v,\varphi)&=&g_1(v,\varphi)-\omega'(v)v_1(v,\varphi), \\
  g_{N,2}(v,\varphi)&=&g_2(v,\varphi)-\omega'(v)\big(v_2(v,\varphi)-\partial_v v_1(v,\varphi) v_1(v,\varphi)\big) - \partial_v g_1(v,\varphi) v_1(v,\varphi)+\omega''(v)v_1^2(v,\varphi).
\end{eqnarray*}

Note that such a transformation is usually applied in the averaging of systems with a small parameter and is associated with a fast variable elimination~\cite{AKN06}. Here, $\varphi$ can serve as an analogue of a fast variable. However, the presence of a small parameter is not assumed in the system, and the terms ``fast'' and ``slow'' variables are not appropriate.

Let us move on to the calculation of the coefficients $v_k(E,\varphi)$. The total derivative of the function $V_N(E,\varphi,t)$ with respect to $t$ along the trajectories of system \eqref{FulSys2} has the following form
\begin{gather}\label{DFS}
    \begin{split}
        \frac{dV_N}{dt}\Big|_{\eqref{FulSys2}} & := \partial_t V_N (E,\varphi,t)+ f (E,\varphi,t)\partial_E V_N (E,\varphi,t)+\Big( \omega(E) +g(E,\varphi,t) \Big)\partial_\varphi V_N (E,\varphi,t)   \\
        & =  \sum_{k=1}^\infty t^{-\frac kq}  \Big(\omega(E) \partial_\varphi v_k(E,\varphi) + f_k(E,\varphi)-\frac{k-q}{q} v_{k-q}(E,\varphi)\Big) \\
        & + \sum_{k=2}^\infty t^{-\frac kq}\sum_{i+j=k} \Big(f_{j}(E,\varphi)\partial_E v_i (E,\varphi) + g_{j}(E,\varphi)\partial_\varphi v_i(E,\varphi) \Big),
    \end{split}
\end{gather}
where it is assumed that  $v_j(E,\varphi)\equiv 0$ for $j\leq 0$ and $j>N$.
Substituting \eqref{VF} into the right-hand side of \eqref{VEq} and the comparison of the result with \eqref{DFS} lead to the following chain of differential equations:
\begin{gather}\label{RSys}
    \omega(E) \partial_\varphi v_k =\Lambda_k(E) - f_k(E,\varphi)+Z_k(E,\varphi), \quad k=1,2,\dots,N,
\end{gather}
where each function $Z_k(E,\varphi)$ is expressed through $v_1,\dots, v_{k-1}$. In particular,
\begin{eqnarray*}
    Z_1&\equiv& 0,\\
    Z_2&\equiv& v_1  \partial_E \Lambda_1 -\big( f_1\partial_E v_1 +  g_1\partial_\varphi v_1\big),\\
    Z_3&\equiv& v_2\partial_E \Lambda_1+ v_1\partial_E\Lambda_2+\frac{1}{2} v_1^2 \partial^2_E\Lambda_1 -\sum_{i+j=3} \big( f_{j}\partial_E v_i +  g_{j}\partial_\varphi v_i\big),\\
    Z_k&\equiv &
    \sum_{
        \substack{
            j+\alpha_1+2\alpha_2+\dots+i \alpha_i=k\\
            \alpha_1+\dots+\alpha_i=m\geq 1
                }
            } C_{i,\alpha,m}   v_1^{\alpha_1}v_2^{ \alpha_2}\dots v_i^{\alpha_i} \partial_E^m\Lambda_j  - \sum_{i+j=k } \big( f_{j}\partial_E v_i +  g_{j}\partial_\varphi v_i\big) + \frac{k-q}{q} v_{k-q},
\end{eqnarray*}
where $C_{i,\alpha,m}={\hbox{\rm const}}$. Define
\begin{gather}\label{Lambda}
    \Lambda_k(E)=\langle f_k(E,\varphi)\rangle-\langle Z_k(E,\varphi)\rangle,
\end{gather}
where
\begin{gather*}
    \langle f_k(E,\varphi)\rangle\stackrel{def}{=}\frac{1}{2\pi}\int\limits_0^{2\pi} f_k (E,\phi)\, d\phi.
\end{gather*}

Hence, for every $k\geq 1$ the right-hand side of \eqref{RSys} is $2\pi$-periodic function with respect to $\varphi$ with zero average. By integrating \eqref{RSys} with respect to $\varphi$, we obtain
\begin{align*}
    v_k(E,\varphi)=  & H_k(X(\varphi,E),Y(\varphi,E))\\
        & +\frac{1}{\omega(E)}\int\limits_{0}^\varphi \Lambda_k(E) - \omega(E)  F_k(X(\phi,E),Y(\phi,E)) \partial_\phi X(\phi,E)+Z_k(E,\phi) \, d\phi
\end{align*}
for $k=1,\dots,N$.
It can easily be checked that each $v_k(E,\varphi)$ is a smooth $2\pi$-periodic function with respect to $\varphi$ such that  $v_k(0,\varphi)\equiv 0$. From \eqref{Lambda} it follows that $\Lambda_k(v)=\mathcal O(v)$ as $v\to 0$.
The function $R_{N+1}(E,\varphi,t)$ has the following form:
\begin{gather*}
R_{N+1}(E,\varphi,t)\equiv \sum_{k=N+1}^\infty t^{-\frac k q}  \Big( f_k-\frac{k-q}{q} v_{k-q}+\sum_{j=1}^{k-1} \big( f_{k-j}\partial_E v_j +  g_{k-j}\partial_\varphi v_j\big)\Big).
\end{gather*}
It is clear that  $R_{N+1}(E,\varphi,t)$ is $2\pi$-periodic functions with respect to $\varphi$ such that $R_{N+1}(0,\varphi,t)\equiv 0$ and $R_{N+1}(v,\varphi,t) =\mathcal O(t^{-(N+1)/q})$ as $t\to\infty$ for all $v\in [0,d_0]$, $\varphi\in \mathbb R$ with $d_0=(1-\sigma)E_0$.

From \eqref{VF} it follows that for all $\sigma\in (0,1)$ there exists $t_0>0$ such that
\begin{gather}\label{Vest}
    (1-\sigma) E\leq V_N(E,\varphi,t)\leq (1+\sigma) E, \quad (1-\sigma) \leq \partial_E V_N(E,\varphi,t)\leq (1+\sigma)
\end{gather}
for all $E\in [0, E_0]$, $\varphi\in \mathbb R$ and $t\geq t_0$. Hence, the transformation $(E,\varphi)\mapsto (v,\varphi)$ is reversible.

Thus, we have
\begin{Lem}
There exists a reversible change of variables $(x,y)\mapsto (v,\varphi)$ which reduces system \eqref{FulSys} into the form \eqref{VEq}, \eqref{PhiEq}.
\end{Lem}

\section{Bifurcations of the equilibrium}
\label{sec3}

In this section, possible bifurcations of the fixed point $(0,0)$ of \eqref{FulSys} as well as the trivial solution of equation \eqref{VEq} are discussed. From the properties of the function $R_{N+1}(v,\varphi,t)$ it follows that the leading terms of asymptotics for solutions of \eqref{VEq} does not depend on $\varphi$. The long-term behaviour of solutions $v(t)$ is determined by the functions $\{\Lambda_k(v)\}_{k=1}^N$. Besides, from \eqref{PhiEq} it follows that $\varphi(t)\to \infty$ as $t\to\infty$, while $v(t) \in [0,d_0]$.

Let $n\geq 1$ be the least natural number such that $\Lambda_n(v)\not\equiv 0$. Then equation \eqref{VEq} takes the form:
\begin{gather*}
    \frac{dv}{dt}=\sum_{k=n}^N t^{-\frac kq} \Lambda_k(v) + R_{N+1}(v,\varphi,t), \quad t\geq t_0.
\end{gather*}

From  \eqref{H0} and \eqref{exch1}, it follows that $E=r^2/2+\mathcal O(r^3)$ as $r\to 0$. Combining this with \eqref{Vest}, we see that $V_N(E,\varphi,t)$ is positive definite function in the vicinity of the fixed point $(0,0)$. Thus, $V_N(E,\varphi,t)$ in the variables $(x,y)$ can be used as a Lyapunov function candidate for system \eqref{FulSys}. If the total derivative of $V_N(E,\varphi,t)$ with respect to $t$ along the trajectories of system \eqref{FulSys2} is sign definite for $E$ close to zero and $\varphi\in\mathbb R$, then this function can be effectively used for the stability analysis of the equilibrium $(0,0)$. It can easily be seen that the right-hand side of \eqref{VEq} coincides with the total derivative of $V_n(E,\varphi,t)$.

\begin{Lem}\label{Lem1}
Let $n\geq 1$ be an integer such that $\Lambda_k(v)\equiv 0$ for $k<n$ and
\begin{gather}
\label{lem1cond}
\Lambda_n(v)=\lambda_n v+\mathcal O(v^2), \quad v\to 0, \quad  \lambda_n={\hbox{\rm const}}\neq 0.
\end{gather}
Then the equilibrium $(0,0)$ of system \eqref{FulSys} is unstable if $\lambda_n>0$ and is stable if $\lambda_n<0$.
Moreover, if $\lambda_n<0$ and $n<q$ ($n=q$), the equilibrium is exponentially (polynomially) stable.
\end{Lem}
\begin{proof}
Consider $V_N(E,\varphi,t)$ with $N=n$ as a Lyapunov function candidate for system \eqref{FulSys}. From \eqref{lem1cond} it follows that the function $v(t)=V_n(E(t),\varphi(t),t)$ satisfies the equations:
\begin{gather*}
    \frac{dv}{dt}=t^{-\frac nq} v (\lambda_n  +\mathcal O(v ) +\mathcal O( t^{-\frac 1q}) )
\end{gather*}
as $t\to\infty$ and $v\to 0$ for all $\varphi\in\mathbb R$. Hence, for all $\sigma\in (0,1)$ there exist $0<d_1\leq d_0$ and $t_1\geq t_0$ such that
\begin{gather}
    \begin{split}\label{estV}
        \frac{dv}{dt}&\geq t^{-\frac nq} (1-\sigma)\lambda_m v\geq 0  \quad \text{if }\ \ \lambda_n>0, \\
        \frac{dv}{dt} &\leq  - t^{-\frac nq} (1-\sigma)|\lambda_m| v\leq 0  \quad \text{if }\ \ \lambda_n<0,
    \end{split}
\end{gather}
for all $v\in [0,d_1]$, $\varphi\in\mathbb R$ and $t\geq t_1$. Integrating the first estimate in \eqref{estV} with respect to $t$ yields the instability of the trivial solution $v(t)\equiv 0$ of equation \eqref{VEq} for all $\varphi\in \mathbb R$.  Indeed, there exists $\epsilon\in (0,d_1/4)$ such that for all $\delta\in (0,\epsilon)$ the solution $v(t)$ with initial data $v(t_1)=\delta$ exceeds the value $\epsilon$ as $t>t_\ast$, where
\begin{eqnarray*}
    t_\ast=t_1 \Big(\frac{2\epsilon}{\delta}\Big)^{\frac{1}{(1-\sigma)\lambda_n}} \quad &\text{if }& \quad \frac{n}{q}=1; \\
    t_\ast^{1-\frac{n}{q}}=t_1^{1-\frac{n}{q}} +\Big(\frac{q-n }{(1-\sigma)\lambda_n q}\Big) \log \Big(\frac{2\epsilon}{\delta}\Big) \quad &\text{if }& \quad \frac{n}{q}\neq 1.
\end{eqnarray*}

Similarly, by integrating the second estimate in \eqref{estV}, we obtain the following inequalities:
\begin{eqnarray}
    \nonumber           0 \leq v(t)\leq v(t_1) \Big(\frac{t}{t_1}\Big)^{-(1-\sigma)|\lambda_n|} \quad &\text{if }& \quad \frac{n}{q}=1,\\
    \label{vestexp}    0 \leq  v(t)\leq v(t_1) \exp\Big(-\frac{(1-\sigma)|\lambda_n| q}{q-n}\big(t^{1-\frac n q}-t_1^{1-\frac n q}\big) \Big) \quad &\text{if }& \quad \frac{n}{q}\neq 1,
\end{eqnarray}
as $t\geq t_1$. From the last estimates it follows that the trivial solution of system \eqref{FulSys2} is stable for all $\varphi\in\mathbb R$. Moreover, the stability is exponential if $n<q$, polynomial if $n=q$ and marginal if $n>q$. Taking into account \eqref{exch1} and \eqref{Vest}, we obtain the corresponding propositions on the stability of the equilibrium $(0,0)$ of system \eqref{FulSys}.
\end{proof}

Note that when the equilibrium loses the stability, the solutions of equation \eqref{VEq} starting from the vicinity of zero either remain inside the domain $(0,d_0)$, or cross the boundary $d_0$ at $t_{\text{exit}}>t_0$.  In the first case, limit cycles may occur. The conditions which guarantee the existence of limit cycles are discussed in the next section. In the second case, the trajectories of \eqref{FulSys} may pass through a separatrix of the limiting system as $t>t_{\text{exit}}$ and can be captured by another attractor. However, such global bifurcations of solutions are not discussed in this paper.

Thus, for non-autonomous perturbations satisfying the conditions of Lemma~\ref{Lem1}, $\lambda_n$ can be considered as a bifurcation parameter such that $\lambda_n=0$ is a critical value.

Let us consider the case when the leading term of the right-hand side of equation \eqref{VEq} is nonlinear with respect to $v$.

\begin{Lem}\label{Lem02}
Let $1\leq m<n$ be integers such that $\Lambda_k(v)\equiv 0$ for $k<m$ and
\begin{gather*}
\Lambda_m(v)=\gamma_{m,s} v^s+\mathcal O(v^{s+1}), \quad \Lambda_j(v)=\mathcal O(v^s), \quad m\leq j<n, \quad \Lambda_{n}(v)=\lambda_n v + \mathcal O(v^2), \quad v\to 0,
\end{gather*}
where $\gamma_{m,s},\lambda_n={\hbox{\rm const}}\neq 0$, $s\in\mathbb Z$, $s\geq2$.
Then the equilibrium $(0,0)$ of system \eqref{FulSys} is
\begin{itemize}
  \item  stable if $\lambda_n<0$ and $\gamma_{m,s}<0$;
  \item  unstable if $\lambda_n>0$ and $\gamma_{m,s}>0$.
\end{itemize}
\end{Lem}
\begin{proof}
As above, consider $V_n(E,\varphi,t)$ as a Lyapunov function candidate. In this case, its total derivative has the form:
\begin{gather*}
\frac{dV_n}{dt}\Big|_{\eqref{FulSys2}}\equiv \frac{dv}{dt}=t^{-\frac m q} \Big[\Lambda_m(v)+\varrho_m(v,\varphi,t)\Big] + t^{-\frac n q} \Big[\Lambda_n(v)+\varrho_n(v,\varphi,t)\Big],
\end{gather*}
where $v(t)=V_n(E(t),\varphi(t),t)$, $|\varrho_m(v,\varphi,t)|\leq M t^{-1/q} v^s$, $|\varrho_n(v,\varphi,t)|\leq M t^{-1/q} v$
as $v\to 0$, $t\to\infty$ for all $\varphi\in\mathbb R$ with $M={\hbox{\rm const}}>0$.
Therefore, for all $\sigma\in (0,1)$ there exist $0<d_1\leq d_0$ and $t_1\geq t_0$ such that
\begin{align*}
\frac{dv}{dt} \leq  -(1-\sigma) \Big(t^{-\frac nq}|\gamma_{n,s}|   v^s+t^{-\frac nq} | \lambda_n| v\Big)\leq 0  \quad &\text{if }\quad \lambda_n<0, \quad  \gamma_{m,s}<0, \\
\frac{dv}{dt} \geq (1-\sigma) \Big( t^{-\frac nq} \gamma_{n,s}  v^s+t^{-\frac nq} \lambda_n v\Big)\geq 0  \quad &\text{if }\quad \lambda_n>0, \quad  \gamma_{m,s}>0,
\end{align*}
for all  $v\in [0,d_1]$, $\varphi\in\mathbb R$ and $t\geq t_1$. From the last estimates it follows that the solution $E(t)\equiv 0$ to system \eqref{FulSys2} is stable if $\lambda_n<0$, $\gamma_{m,s}<0$ and unstable if $\lambda_n>0$, $\gamma_{m,s}>0$. Combining \eqref{exch1} and  \eqref{Vest}, we obtain the corresponding propositions on the stability of the equilibrium $(0,0)$ of system \eqref{FulSys}.
\end{proof}

Note that in some cases the last proposition can be improved. In particular, we have

\begin{Lem}\label{Lem2}
Under the conditions of Lemma~\ref{Lem02}, we have
\begin{itemize}
  \item in case $m<n<q$, the equilibrium $(0,0)$ of system \eqref{FulSys} is
  \begin{itemize}
    \item  exponentially stable if $\lambda_n<0$;
    \item   polynomially stable if $\lambda_n>0$ and $\gamma_{m,s}<0$.
  \end{itemize}
  \item  in case $m<n=q$,  the equilibrium $(0,0)$ of system \eqref{FulSys} is polynomially stable if  $ \lambda_n+\frac{n-m}{q(s-1)}<0$ or $ \lambda_n+\frac{n-m}{q(s-1)}>0$ and $\gamma_{m,s}<0$;
  \item in case $m<q<n$, the equilibrium $(0,0)$ of system \eqref{FulSys} is polynomially stable if $\gamma_{m,s} <0$;
  \item in case $q \leq m<n$, the equilibrium $(0,0)$ of system \eqref{FulSys} is stable if $\gamma_{m,s} <0$.
\end{itemize}
\end{Lem}
\begin{proof}
It can easily be checked that the derivative of the function $v(t)=V_n(E(t),\varphi(t),t)$ satisfies the asymptotic estimate:
\begin{gather*}
    \frac{dV_n}{dt}\Big|_{\eqref{FulSys2}}\equiv \frac{dv}{dt}=t^{-\frac mq}v^s(\gamma_{m,s}+\mathcal O(v) +\mathcal O(t^{-\frac 1q}))+t^{-\frac nq}v (\lambda_n +\mathcal O(v) +\mathcal O(  t^{-\frac{ 1}{q}}) )
\end{gather*}
 as $t\to\infty$, $v\to 0$ for all $\varphi\in\mathbb R$. The right-hand side of the last expression is not sign definite uniformly for all small $v$ and big $\tau$. Indeed, if $v\sim \epsilon$ and $t\sim \epsilon^{-\kappa}$,  where $0<\epsilon\ll 1$, $\kappa={\hbox{\rm const}}$, then the sign of $dv/dt$ is determined by $\lambda_n$ in case $\kappa<q(s-1)/(n-m)$, and by $\gamma_{m,s}$ in the opposite case. Therefore, $V_n(E,\varphi,t)$ can not be used as a Lyapunov function for system \eqref{FulSys}.

Define
\begin{gather*}
    U(E,\varphi,t)\equiv t^\nu V_n(E,\varphi,t), \quad \nu=\frac{n-m}{q(s-1)}>0.
\end{gather*}
From \eqref{Vest} it follows that
\begin{gather}\label{Uineq}
    (1-\sigma) t^{\nu} E\leq U(E,\varphi,t)\leq (1+\sigma) t^{\nu} E
\end{gather}
for all $E\in [0,E_0]$, $\varphi\in \mathbb R$ and $t\geq t_0$. This function corresponds to the change of variable $u(t)=t^{\nu} v(t)$ such that equation \eqref{VEq} takes the form:
\begin{eqnarray}\label{Ueq}
    \frac{du}{dt} &=&\nu t^{-1}+\sum_{k=m}^n t^{\nu-\frac kq} \Lambda_k\big(t^{-\nu} u\big) + t^{\nu} R_{n+1}\big(t^{-\nu}u,\varphi,t\big).
\end{eqnarray}
Hence the total derivative of the function $U(E,\varphi,t)$ with respect to $t$ along the trajectories of system \eqref{FulSys2} has the following asymptotics:
 \begin{gather}\label{Uest}
    \frac{dU}{dt}\Big|_{\eqref{FulSys2}}=\frac{du}{dt}=u \Big(\nu t^{-1}  +  t^{-\frac nq} (\lambda_n+\gamma_{m,s}u^{s-1}) +\mathcal O(t^{-\frac{n+q\nu}{q}})+\mathcal O( t^{-\frac{n+1}{q}})\Big)
\end{gather}
as $t\to\infty$ for all $u\in [0, U_0]$ and $\varphi\in\mathbb R$ with $U_0={\hbox{\rm const}}>0$.

Consider the case $m<n<q$. From \eqref{Uest} it follows that
\begin{gather*}
    \frac{du}{dt}=u  t^{-\frac nq}\Big(  \lambda_n+\gamma_{m,s}u^{s-1} +\mathcal O(t^{-\nu}) +\mathcal O(t^{-\frac 1q})\Big)
\end{gather*}
as $t\to\infty$. If $\lambda_n<0$, then for all  $\sigma\in (0,1)$ there exist $0<U_1\leq U_0$ and $t_1\geq t_0 $ such that
\begin{gather*}
    \frac{du}{dt}\leq  - t^{-\frac nq} (1-\sigma)|\lambda_n| u\leq 0
\end{gather*}
for all $u\in [0,U_1]$, $\varphi\in\mathbb R$ and $t\geq t_1$. Integrating the last inequality with respect to $t$, we get an estimate of the form \eqref{vestexp} as $t\geq t_1$. Combining this with \eqref{Uineq}, we get exponential stability of the solution $E(t)\equiv 0$ of equation \eqref{FulSys2}.

If $\lambda_n>0$ and $\gamma_{m,s}<0$, the leading term of $du/dt$ has a zero at $u=U_c$, $U_c=(\lambda_n/|\gamma_{m,s}|)^{1/(s-1)}$. Let us show that $U(E(t),\varphi(t),t)\to U_c$ as $t\to\infty$. Consider the change of variable $u(t)=U_c+z(t)$ in equation \eqref{Ueq}. Then $z(t)$ satisfies the equation:
\begin{gather}
    \label{zeq0}
    \frac{dz}{dt} =  t^{-\frac nq} (U_c+z) \Big(\lambda_n+\gamma_{m,s} (U_c+z)^{s-1} \Big)+ p(z,\varphi,t),
\end{gather}
where $|p(z,\varphi,t)|\leq M t^{-n/q} [t^{-\nu}+ t^{-1/q}]$ as $t\geq t_1$ for all $|z|\leq z_1$ and $\varphi\in\mathbb R$ with positive constants $M$ and $z_1$. It can easily be checked that the unperturbed equation with $p(z,\varphi,t)\equiv 0$ has asymptotically stable solution $z(t)\equiv 0$.
Let us show that this solution is stable with respect to the perturbation $p(z,\varphi,t)$. Consider $\ell(z)=z^2/2$ as a Lyapunov function candidate for equation \eqref{zeq0}. The total derivative of $\ell(z)$ has the form:
\begin{gather*}
\frac{d\ell}{dt}\Big|_{\eqref{zeq0}} =  t^{-\frac nq}  z\Big(-|\gamma_{m,s}| (s-1)U_c^{s-1} z + \mathcal O(z^2)+\mathcal O(t^{-\nu})+ \mathcal O(t^{-\frac 1q})\Big)
\end{gather*}
as $z\to 0$ and $t\to\infty$. Therefore, for all  $\sigma>0$  there exist $0<z_2\leq z_1$ and $t_2\geq t_1$ such that
\begin{gather*}
\frac{d\ell }{dt}\Big|_{\eqref{zeq0}} \leq   t^{-\frac nq}  \Big( -2(1-\sigma)|\gamma_{m,s}| (s-1)U_c^{s-1}  \ell +M z_2 [t^{-\nu}+ t^{-1/q}]\Big)
\end{gather*}
 for all $|z|\leq z_2$, $\varphi\in \mathbb R$ and $t\geq t_2$. By integrating the last inequality, we get
$\ell(z(t))=\mathcal O(t^{-\nu}) +\mathcal O(t^{-1/q})$ as $t\to\infty$ for solutions with initial data $|z(t_2)|\leq z_2$. Hence, $U(E(t),\varphi(t),t)=U_c+\mathcal O(t^{-\nu/2}) +\mathcal O(t^{-1/2q})$ and $E(t)=\mathcal O(t^{-\nu})$ as $t\to\infty$. Therefore, the solution $E(t)\equiv 0$ of \eqref{FulSys2} is polynomially stable.

Consider the case $m<n=q$. It follows from \eqref{Uest} that
\begin{gather*}
 \frac{du}{dt}=u  t^{-1}\Big( \lambda_n+ \nu+\gamma_{m,s}u^{s-1} +\mathcal O(t^{-\nu})+\mathcal O(t^{-\frac 1q})\Big)
\end{gather*}
as $t\to\infty$. If $\lambda_n+\nu<0$, then for all $\sigma>0$ there exist $0<U_1\leq U_0$ and $t_1\geq t_0$ such that
\begin{gather*}
\frac{du}{dt}\leq   t^{-1} (-|\lambda_n+\nu|+\sigma)u\leq 0
\end{gather*}
for all $u\in [0, U_1]$, $\varphi\in\mathbb R$ and $t\geq t_1$. Integrating the last inequality yields $ 0\leq v(t)=t^{-\nu} u(t)\leq v(t_1)  (t/t_1)^{\lambda_n+\sigma} $ as $t\geq t_1$. Therefore, the solution $E(t)\equiv 0$ is polynomially stable. If $\lambda_n+\nu>0$ and $\gamma_{m,s}<0$, then, as above, there is a family of solutions such that $U(E(t),\varphi(t),t)=U_\nu+\mathcal O(t^{-\nu/2}) +\mathcal O(t^{-1/2q})$ as $t\to\infty$, where $U_\nu=((\lambda_n+\nu)/|\gamma_{m,s}|)^{1/(s-1)}$. Taking into account the transformation of variables, we get polynomial stability of the solution $E(t)\equiv 0$ to system \eqref{FulSys2}.

In the case $m<q<n$, the function $U(E,\varphi,t)$ can not be used in the stability analysis. Consider the function
\begin{gather*}
    W(E,\varphi, t)\equiv t^\eta V_n(E,\varphi,t),\quad \eta=\frac{q-m}{q(s-1)}>0,
\end{gather*}
which corresponds to the change of variables in \eqref{VEq}: $v(t)=t^{-\eta} w(t)$. The total derivative of $W(E,\varphi, t)$ has the asymptotics
\begin{gather*}
 \frac{dw}{dt}=t^{-1} w \Big( \eta +\gamma_{m,s} w^{s-1}+\mathcal O(t^{-\beta})+ \mathcal O(t^{-\frac{n-m}{q}}) \Big)
\end{gather*}
as $t\to\infty$ for all $w\in[0,W_0]$ and $\varphi\in \mathbb R$ with $W_0={\hbox{\rm const}}>0$. If $\gamma_{m,s}<0$, then, as above, system has a family of solutions such that $W(E(t),\varphi(t),t)=W_c+\mathcal O(t^{-\eta/2})+ \mathcal O(t^{-(n-m)/2q})$ as $t\to\infty$, where $W_c=(\eta/|\gamma_{m,s}|)^{1/(s-1)}$. Hence, $E(t)=\mathcal O(t^{-\eta})$ as $t\to\infty$.

Finally, consider the case $q\leq m<n$. For all $\sigma\in (0,1)$ there exist $0<d_1\leq d_0$ and $t_1\geq t_0 $ such that
\begin{gather*}
\frac{dv}{dt}\leq   - (1-\sigma) |\gamma_{m,s}| t^{-\frac m q}v^s+M t^{-\frac n q} v
\end{gather*}
for all  $v\in [0,d_1]$, $\varphi\in\mathbb R$ and $t\geq t_1$.
Let us fix $0<\epsilon<d_1$ and define
\begin{gather*}
t_2=\max\Big\{\Big(\frac{2^s M}{\epsilon^{s-1} (1-\sigma) |\gamma_{m,s}|}\Big)^{\frac{q}{n-m}},t_1\Big\}.
\end{gather*}
Then for all $v\in (\epsilon/2,\epsilon)$ and $t\geq t_2$ we have
 \begin{gather*}
\frac{dv}{dt}\leq -  t^{-\frac mq} \frac{ (1-\sigma)|\gamma_{m,s}| }{2} v^s.
\end{gather*}
This implies that any solution $v(t)$ of equation \eqref{VEq} with initial data $0<v(t_2)<\epsilon/2$ can not exceed the value $\epsilon$ as $t\geq t_2$. Hence, the solution $E(t)\equiv 0$ is at least neutrally stable.
\end{proof}

Let us remark that in the case $n=q$ with $\gamma_{m,s}>0$ and $\lambda_n+\nu>0$, Lyapunov stability of the equilibrium is not justified. Moreover, from Lemma~\ref{Lem2} it follows that the trivial solution is weakly unstable with the weight $t^\nu$: there exists $\epsilon>0$ such that for arbitrarily small initial data $\exists\, t_\ast>0$: $E(t)t^\nu\geq \epsilon$ as $t\geq t_\ast$. From \eqref{H0} it follows that the equilibrium $(0,0)$ is unstable with the weight $t^{\nu/2}$. Similarly, in the case $m<q<n$ with $\gamma_{m,s}>0$, the fixed point $(0,0)$ of system \eqref{FulSys} is unstable with the weight $t^{\eta/2}$.

For non-autonomous perturbations satisfying the conditions of Lemma~\ref{Lem02}, the stability of the equilibrium is determined by two parameters $\lambda_n$ and $\gamma_{m,s}$ (see~Fig.~\ref{Fig0}). The partition of the parameter plane depends on the ratio $n/q$. Note that if $\lambda_n>0$, the equilibrium becomes unstable in the corresponding linearized system. However, the asymptotic stability can preserve in the complete system due to nonlinear terms of equations. In this case, the system has a Hopf bifurcation in the scaled variables.

\begin{figure}
\centering
\subfigure[$n<q$]{\includegraphics[width=0.3\linewidth]{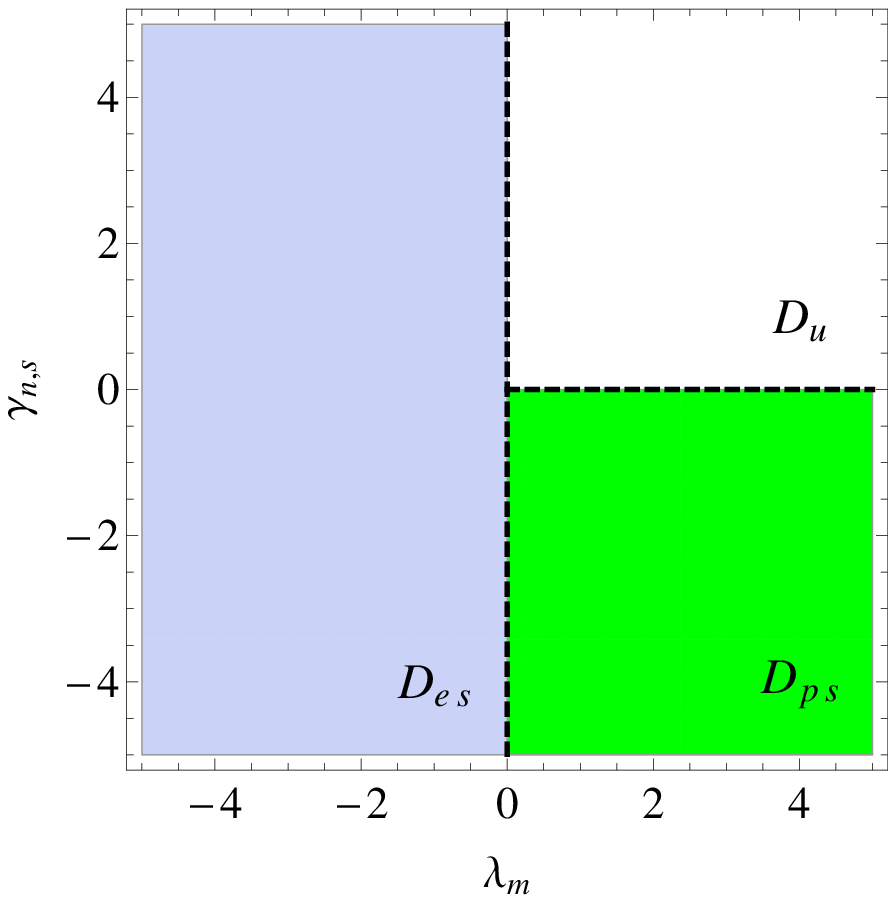}} \hspace{3ex}
\subfigure[$n=q$]{\includegraphics[width=0.3\linewidth]{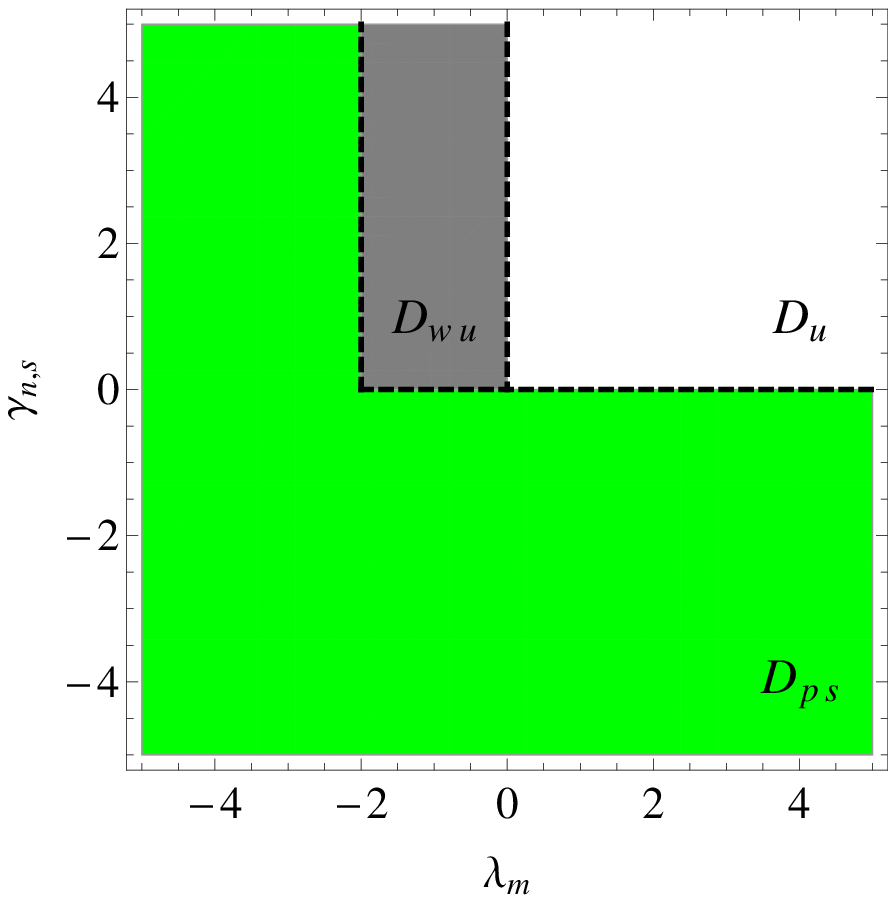}} \hspace{3ex}
\subfigure[$q<n$]{\includegraphics[width=0.3\linewidth]{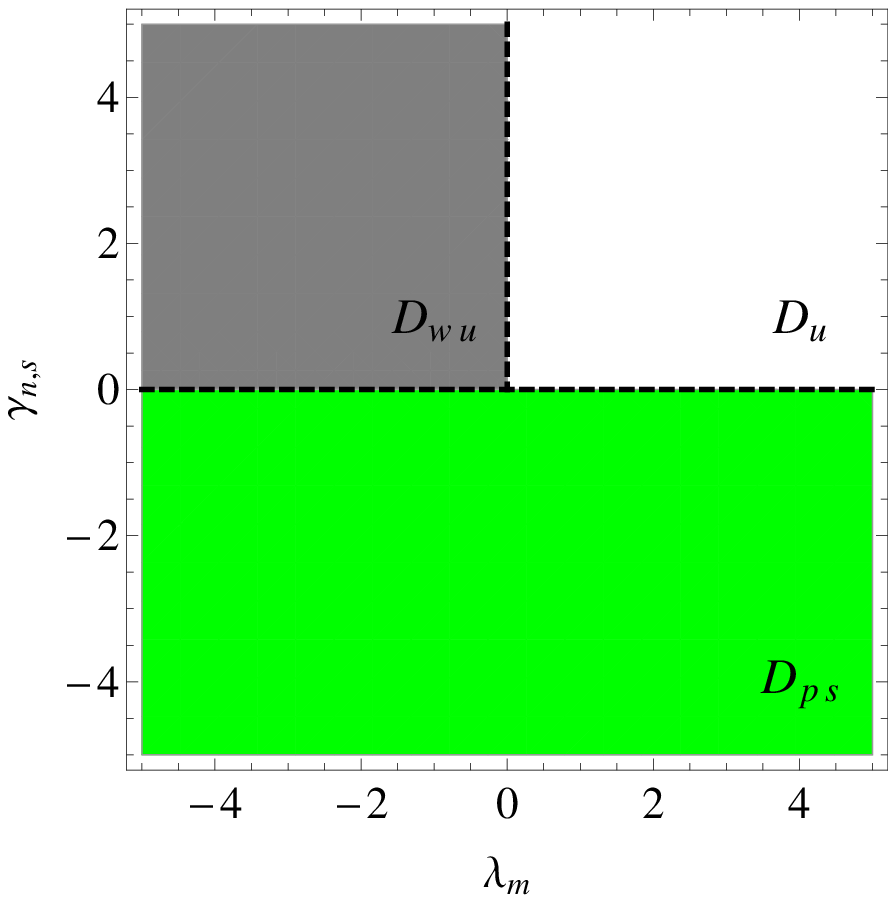}}
\caption{\small Partition of the parameter plane $(\lambda_n,\gamma_{m,s})$ when $m<q$. Here, $D_{e s}$ and $D_{p s}$ are the domains of exponential and polynomial stability, $D_{u}$ is the domain of instability and $D_{w u}$ is the domain of instability with a weight. } \label{Fig0}
\end{figure}

Now let us consider the case when the right-hand side of equation \eqref{VEq} has no linear terms in $v$.

\begin{Lem}\label{Lem22}
Let  $1\leq m<n$ be integers such that $\Lambda_k(v) \equiv 0$ for $k<m$ and
\begin{gather*}
\Lambda_m(v)  =  \gamma_{m,s} v^s+\mathcal O(v^{s+1}), \quad \Lambda_j(v)=\mathcal O(v^s), \quad j<n,\\
\Lambda_n(v)=\gamma_{n,d} v^d + \mathcal O(v^{d+1}), \quad \Lambda_i(v)=\mathcal O(v^d), \quad i>n
\end{gather*}
as  $v\to 0$ with $\gamma_{m,s},\gamma_{n,d}={\hbox{\rm const}}\neq 0$, $s,d\in\mathbb Z$, $s,d\geq2$.
\begin{enumerate}
\item In case  $s\leq d$, the equilibrium $(0,0)$ of system \eqref{FulSys} is
        \begin{itemize}
            \item stable if $\gamma_{m,s}<0$;
            \item unstable if $\gamma_{m,s}>0$.
        \end{itemize}
\item In case  $s> d$, the equilibrium $(0,0)$ of system \eqref{FulSys} is
        \begin{itemize}
            \item  stable if  $\gamma_{m,s}<0$ and $\gamma_{n,d}<0$;
            \item unstable if $\gamma_{m,s}>0$ and $\gamma_{n,d}>0$.
        \end{itemize}
    \end{enumerate}
\end{Lem}
\begin{proof}
If $s\leq d$, the total derivative of the function $V_n(E,\varphi,t)$ has the asymptotics:
\begin{gather*}
    \frac{dV_n}{dt}\Big|_{\eqref{FulSys2}}\equiv \frac{dv}{dt}=t^{-\frac nq}  v^s (\gamma_{n,s}+\mathcal O(v)+\mathcal O(t^{-\frac 1q}))
\end{gather*}
as $t\to\infty$ and $v\to 0$ for all $\varphi\in\mathbb R$. If $\gamma_{n,s}>0$, the total derivative is locally positive and the solution $E(t)\equiv 0$ of system \eqref{FulSys2} is unstable. In the opposite case, when  $\gamma_{n,s}<0$, the total derivative is locally negative and the trivial solution is stable.

Let $s>d$, $\gamma_{m,s}<0$  and $\gamma_{n,d}<0$. Then for all $\sigma>0$ there exist $0<d_1\leq d_0$ and $t_1\geq t_0 $ such that
\begin{gather*}
\frac{dv}{dt}\leq  - t^{-\frac nq} (1-\sigma)v^{d} \big(|\gamma_{m,s}| v^{s-d}+|\gamma_{n,d}|t^{-\frac{n-m}{q}}\big)\leq  0
\end{gather*}
for all $v\in [0,d_1]$, $t\geq t_1$ and $\varphi\in\mathbb R$. Hence, the solution $E(t)\equiv 0$ to system \eqref{FulSys2} is stable. Similarly, if $\gamma_{m,s}>0$  and $\gamma_{n,d}>0$, the trivial solution is unstable.
\end{proof}

\section{Bifurcations of limit cycles}
\label{sec32}

 Let us show that decaying non-autonomous bifurcations may lead to the appearance of limit cycles.

\begin{Lem}\label{Lem3}
Let $n\geq 1$ be an integer such that $\Lambda_k(v)\equiv 0$ for $k<n$,  $\Lambda_n(v)\not\equiv 0$
and $V_c\in (0,E_0)$ be a real number such that $\Lambda_n(V_c)=0$ and $\Lambda_n'(V_c)< 0$.
Then for all $\epsilon>0$ there exist $\delta_\ast>0$ and $t_\ast>0$ such that $\forall\, (x_0,y_0)$: $|H_0(x_0,y_0)-V_c|\leq \delta_\ast$ and $\tau_0\geq t_\ast$ the solution $x(t)$, $y(t)$ of system \eqref{FulSys} with initial data $x(\tau_0)=x_0$, $y(\tau_0)=y_0$ satisfies the estimate $|H_0(x(t),y(t))- V_c|<\epsilon$ for all $t\geq \tau_0$. Moreover, if $1\leq n\leq q$, $H_0(x(t),y(t))\to V_c $ as $t \to \infty$.
\end{Lem}
\begin{proof}
Consider the function $V_n(E,\varphi,t)$ on the trajectories of system \eqref{FulSys2}. It follows that  $v(t)=V_n(E(t),\varphi(t),t)$ satisfies the following equation:
\begin{gather*}
\frac{dv}{dt}=t^{-\frac nq}\Big( \Lambda_n(v) + \mathcal O(t^{-\frac 1q})\Big)
\end{gather*}
as $t\to\infty$ for all $v\in [0, d_0]$ and $\varphi\in \mathbb R$. The change of the variable $v(t)=V_c+z(t)$ leads to the following equation:
\begin{gather}
\label{zeq}
\frac{dz}{dt}=t^{-\frac nq}  \Lambda_n(V_c+z) + p(z,\varphi,t),
\end{gather}
where $p(0,\varphi,t)\not\equiv 0$, $|p(z,\varphi,t)|\leq M t^{-{(n+1)}/{q}}$  for all $z\in [-V_c,d_0+V_c] $, $\varphi\in \mathbb R$ and $t\geq t_0$ with positive constant $M>0$.
From $\lambda_n:=\Lambda_n'(V_c)<0$ it follows that the trivial solution of \eqref{zeq} with $p(z,\varphi,t)\equiv 0$ is stable. Let us show that this solution is stable with respect to the perturbation $p(z,\varphi,t)$. Indeed, consider $\ell(z)=z^2/2$ as a Lyapunov function candidate for \eqref{zeq}. Its total derivative has the form:
\begin{gather}
\label{elleq}
\frac{d\ell}{dt}\Big|_{\eqref{zeq}}=t^{-\frac{n}{q}}\Lambda_n(V_c+z)z+p(z,\varphi,t)z.
\end{gather}
First note that there exists $\delta_1>0$ such that $\Lambda_n(V_c+z)z\leq -|\lambda_n| z^2/2$ as $|z|\leq \delta_1$. Let us fix $0<\epsilon<\delta_1$ and choose
\begin{gather*}
\delta_\ast=\frac{\epsilon}{2}, \quad t_\ast=\max\Big\{\Big(\frac{4M}{\delta_\ast |\lambda_n|}\Big)^q, t_0\Big\},
\end{gather*} then
\begin{gather*}
\frac{d\ell}{dt}\leq -t^{-\frac{n}{q}} z^2 \Big(\frac{|\lambda_n|}{2}-M\delta_\ast^{-1} t_\ast^{-\frac{1}{q}} \Big)\leq -t^{-\frac{n}{q}} z^2\frac{|\lambda_n|}{4}
\end{gather*}
for all $\delta_\ast\leq |z|\leq \epsilon$, $\varphi\in\mathbb R$ and $t\geq t_\ast$.
Hence, any solution $z(t)$ with initial data $|z(\tau_0)|\leq \delta_\ast$, $\tau_0\geq t_\ast$ cannot leave the domain $\{|z|<\epsilon\}$ as $t\geq \tau_0$. Returning to the original variables, we obtain the result of the Lemma.

Let us show that $H_0(x(t),y(t))\to V_c$ as $t\to\infty$ if $n\leq q$. From \eqref{elleq} it follows that $d\ell/dt\leq t^{-n/q}(-|\lambda_n|\ell + M \delta_1 t^{-1/q})$ as $t\geq t_0$. By integrating the last inequality in the case $n=q$, we get
\begin{gather*}
0\leq \ell(z(t))\leq \ell (z(t_0)) \Big(\frac{t}{t_0}\Big)^{-|\lambda_n|}+M\delta_1 t^{-|\lambda_n|}\int\limits_{t_0}^t s^{|\lambda_n|-\frac{1}{q}-1}\, ds
\end{gather*}
with $|z(t_0)|\leq \delta_1$. Similar estimates hold in the case $n<q$. Hence, if $n\leq q$, $v(t)\to V_c$ as $t\to\infty$.
\end{proof}

\begin{Lem}
Let $1\leq n\leq q$ be an integer such that $\Lambda_k(v)\equiv 0$ for $k<n$, $\Lambda_n(v)\not\equiv 0$ and $V_c\in (0,E_0)$ be a real number such that $\Lambda_n(V_c)=0$ and $\Lambda_n'(V_c)> 0$. Then there exists $\epsilon>0$ such that for all $ \delta_\ast>0$ $\exists\, (x_0,y_0)$: $|H_0(x_0,y_0)-V_c|\leq \delta$ and $\tau_0>0$ the solution $x(t)$, $y(t)$ of system \eqref{FulSys} with initial data $x(\tau_0)=x_0$, $y(\tau_0)=y_0$ satisfies the estimate $|H_0(x(t),y(t))- V_c|\geq \epsilon$ at some $t\geq \tau_0$.
\end{Lem}
\begin{proof}
From \eqref{elleq} it follows that
\begin{gather*}
\frac{d\ell}{dt}\geq t^{-\frac{n}{q}} z^2 \Big(\frac{|\lambda_n|}{2}-M\delta_\ast^{-1} \tau_0^{-\frac 1q}\Big)
\end{gather*}
for all $\delta_\ast\leq |z|<\delta_1$, $\varphi\in\mathbb R$ and $t\geq \tau_0$. We choose $\tau_0=(2M\delta_\ast/|\lambda_n|)^q$ so that $d\ell/dt\geq t^{-n/q}\delta_\ast^2|\lambda_n|/4$. By integrating the last inequality in the case $n=q$, we get $\ell(z(t))\geq \delta_\ast^2/4(2+|\lambda_n| \log(t/\tau_0))$, where $z(\tau_0)=\delta_\ast$. Thus, there exists $0<\epsilon<\delta_1$ such that for all $\delta_\ast<\epsilon$ the solution $z(t)$ satisfies the estimate $|z(t)|>\epsilon$ at
\begin{gather*}
t=2\tau_0 \exp\Big(\frac{2\epsilon^2}{\delta_\ast^2 (2+|\lambda_n|)}\Big).
\end{gather*}
Similar estimate holds in the case $n<q$.
\end{proof}

\begin{Cor}
Let $1\leq n\leq q$ be an integer such that in $\Lambda_k(v)\equiv 0$ for $ k<n$, $\Lambda_n(v)\not\equiv 0$ and $V_c^l\in (0,E_0)$, $l=1,\dots,s$, $s\geq 1$ be a set of real numbers such that $\Lambda_n(V_c^l)=0$, $\Lambda_n'(V_c^l)< 0$. Then there exist a set $\{(x,y)\in\mathbb R^2: H_0(x,y)\equiv V_c^l, l=1,\dots,s\}$ of stable limit cycles of system \eqref{FulSys}.
\end{Cor}

\section{Applications}
\label{sec4}

In the previous two sections, we described  possible asymptotics regimes in systems of the form \eqref{FulSys}. In this section, we give some conditions on the perturbations which guarantee the applicability of these results.

\begin{Th}
\label{Th1}
Let $l, h, n$ be positive integers such that $l+h\geq n$ and the coefficients of the perturbations \eqref{HF} satisfy the following conditions:
\begin{align}
    \label{cond00}
&        H_i(x,y)\equiv 0, \quad 1\leq i<h,\quad F_j(x,y)\equiv 0, \quad 1\leq j<l; \\
    \label{cond01}
&       \oint\limits_{H_0(x,y)=E} F_{k}(x,y) \partial_y H_0(x,y)\, dl =0, \quad \forall\, E \in (0,E_0), \quad l\leq k< n;\\
    \label{cond02}
&       F_{k}(x,y)=\mathcal O(r^2), \quad l\leq k< n, \quad  F_n(x,y)=\lambda_n y+\mathcal O(r^2), \quad r\to 0,
\end{align}
where $\lambda_n=\text{const}\neq 0$. Then the equilibrium $(0,0)$ of system \eqref{FulSys} is unstable if $\lambda_n>0$ and is stable if $\lambda_n<0$.
Moreover, if $\lambda_n<0$ and $n<q$ ($n=q$), the equilibrium is exponentially (polynomially) stable.
\end{Th}
\begin{proof}
Let us show that there exists a transformation $(x,y)\mapsto (E,\varphi)\mapsto (v,\varphi)$ such that equation \eqref{VEq} has $\Lambda_k(v)\equiv 0$ for $k<n$ and $\Lambda_n(v)=\lambda_n v + \mathcal O(v^2)$ as $v\to 0$. In this case, Lemma~\ref{Lem1} is applicable.

To be definite, let $h<l$. The proof in the case $h\geq l$ is similar. From \eqref{cond00} it follows that $f_k(E,\varphi)\equiv 0$, $g_k(E,\varphi)\equiv 0$ for $1\leq k<h$;
\begin{eqnarray*}
    f_k(E,\varphi)\equiv -\omega(E) \partial_\varphi H_{k}(X(\varphi,E),Y(\varphi,E)), \ \
    g_k(E,\varphi)\equiv \omega(E) \partial_E H_k(X(\varphi,E),Y(\varphi,E))
\end{eqnarray*}
for $h\leq k<l$.
The functions $f_k(E,\varphi)$, $g_k(E,\varphi)$ have the form \eqref{fg} for $l\leq k\leq n$. Hence, a Lyapunov function candidate for system \eqref{FulSys} should be considered in the following form:
\begin{gather}
\label{VFm-l}
V(E,\varphi,t)=E+\sum_{k=h}^n t^{-\frac k q} v_k(E,\varphi).
\end{gather}
We assume that $v_i(E,\varphi)\equiv 0$ for $i<h$. According to the scheme described in section~\ref{sec2}, the functions  $v_k(E,\varphi)$ are determined from system \eqref{RSys}.
Therefore, for $h\leq k<l$, we have
\begin{gather*}
Z_{k}(E,\varphi)\equiv 0, \quad \Lambda_{k}(E)\equiv \langle f_{k} (E,\varphi)\rangle \equiv 0, \quad  v_{k}(E,\varphi)\equiv H_{k}(X(\varphi,E),Y(\varphi,E)).
\end{gather*}

For $l\leq k\leq n-1$, condition \eqref{cond01} is used to guarantee the equality $\Lambda_k(E)\equiv0$. Indeed, from \eqref{DFS} it follows that
\begin{eqnarray*}
Z_{l}(E,\varphi)&=&-\sum_{i+j=l}\Big(  f_{i}(E,\varphi)\partial_E v_{j} (E,\varphi) + g_{i}(E,\varphi) \partial_\varphi v_{j}(E,\varphi)\Big)\\
&=&-\sum_{j=h}^{l-h}\Big(   f_{l-j}(E,\varphi)\partial_E v_{j} (E,\varphi)+ g_{l-j}(E,\varphi)\partial_\varphi v_{j}(E,\varphi) \Big)\\
& = &\omega(E)\sum_{j=h}^{l-h}\Big( \partial_\varphi H_{l-j}\partial_E H_{j} -\partial_E H_{l-j}\partial_\varphi H_j  \Big) \equiv 0.
\end{eqnarray*}
Hence,
\begin{align*}
 \Lambda_{l}(E)\equiv   \langle f_{l} (E,\varphi)\rangle &\equiv \omega(E) \langle   F_{l}(X(\varphi,E),Y(\varphi,E))  \partial_\varphi X (\varphi,E)\rangle \\
 &\equiv \langle   \partial_y H_0(X(\varphi,E),Y(\varphi,E))   F_{l}(X(\varphi,E),Y(\varphi,E)) \rangle \equiv 0
\end{align*}
and $v_{l}(E,\varphi)\equiv H_{l}(X(\varphi,E),Y(\varphi,E))+\hat v_l(E,\varphi)$, where
\begin{gather*}
 \hat v_l(E,\varphi)=\int  F_{l}(X(\varphi,E),Y(\varphi,E)) \partial_\varphi X (\varphi,E) \, d\varphi.
\end{gather*}

For $l +1\leq k\leq n-1$, we have
\begin{eqnarray*}
Z_{k}(E,\varphi)&=&-\sum_{i+j=k}\Big(  f_{i}(E,\varphi) \partial_E v_{j} (E,\varphi)+ g_{i}(E,\varphi)\partial_\varphi v_{j}(E,\varphi) \Big).
\end{eqnarray*}
Note that $f_j$ and $g_j$ with $j\geq l$ are not involved in these sums.  Such functions have the multipliers $\partial_\varphi v_{k-j}$ and $\partial_E v_{k-j}$, correspondingly. If $j \geq l$, then $k-j\leq n-j-1\leq n-l-1\leq h-1$ and $v_{k-j}(E,\varphi)\equiv 0$. Similarly, the functions $\partial_\varphi v_{j}$ and $\partial_\varphi v_{j}$ with $j\geq l$ are not contained in these sums. Therefore,
\begin{eqnarray*}
Z_{k}(E,\varphi)&=&-\sum_{\substack{i+j=k\\ h\leq i,j< l } } \Big( f_{i}(E,\varphi) \partial_E v_{j} (E,\varphi) + g_{i}(E,\varphi)\partial_\varphi v_{j}(E,\varphi) \Big)\\
&=&-\sum_{j=h}^{k-h}\Big(  f_{k-j}(E,\varphi) \partial_E v_{j} (E,\varphi)+ g_{k-j}(E,\varphi) \partial_\varphi v_{j}(E,\varphi) \Big)\\
& = &\omega(E)\sum_{j=h}^{k-h}\Big(\partial_\varphi H_{k-j}\partial_E H_{j}  -\partial_E H_{k-j}\partial_\varphi H_j  \Big) \equiv 0
\end{eqnarray*}
and
$\Lambda_k(E)=0$. This implies that $v_{k}(E,\varphi)\equiv H_{k}(X(\varphi,E),Y(\varphi,E))+\hat v_k(E,\varphi)$, where
\begin{gather*}
  \hat v_k(E,\varphi)=\int  F_{k}(X(\varphi,E),Y(\varphi,E)) \partial_\varphi X(\varphi,E) \, d\varphi.
\end{gather*}
We see that $v_k(E,\varphi)=\mathcal O(E)$, $\hat v_k(E,\varphi)=\mathcal O(E^{3/2})$ as $E\to 0$ for all $\varphi\in\mathbb R$.

For $k=n$, the function $Z_n(E,\varphi)$ has the following form:
\begin{eqnarray*}
Z_{n}(E,\varphi)&=&-\sum_{\substack{i+j=n\\ h\leq i,j\leq l } } \Big(  f_{i}(E,\varphi)\partial_E v_{j} (E,\varphi) +g_{i}(E,\varphi) \partial_\varphi v_{j}(E,\varphi) \Big)\\
&=& -\sum_{\substack{i\in \{l,n-l\} }} \Big( f_{i}(E,\varphi) \partial_E v_{n-i} (E,\varphi) + g_{i}(E,\varphi)\partial_\varphi v_{n-i}(E,\varphi) \Big)\\
& & -\sum_{\substack{i+j=n\\ h\leq i,j\leq l-1 } } \Big(  f_{i}(E,\varphi)\partial_E v_{j} (E,\varphi) + g_{i}(E,\varphi)\partial_\varphi v_{j}(E,\varphi) \Big)\\
&\equiv & Z_n^1(E,\varphi) + Z_n^2(E,\varphi).
\end{eqnarray*}
If  $l+h>n$, then $Z_n^1(E,\varphi)\equiv 0$. If $l+h=n$,
\begin{gather}
\label{zn1}
Z_n^1(E,\varphi)  = \omega(E)\Big(\partial_\varphi H_{h} (\partial_E \hat v_l +  F_l \partial_E X)-\partial_E H_{h}(\partial_\varphi \hat v_l + F_l \partial_\varphi X )\Big).
\end{gather}
It can easily be checked that
\begin{gather}
\label{zn2}
Z_n^2(E,\varphi)= \omega(E)\sum_{\substack{i+j=n\\ h\leq i,j\leq l-1 }}\Big( - \partial_\varphi H_j  \partial_E H_{i} +\partial_E H_j  \partial_\varphi H_{i}\Big) \equiv 0.
\end{gather}
Consequently,
\begin{gather*}
\Lambda_n(E)=\omega(E)\langle  F_n(X(\varphi,E),Y(\varphi,E)) \partial_\varphi X(\varphi,E) - Z_n(E,\varphi)\rangle
\end{gather*}
and $v_n(E,\varphi)=H_n(E,\varphi)+\hat v_n(E,\varphi)$, where
\begin{gather*}
\hat v_n(E,\varphi)= \frac{1}{\omega(E)}\int \Big( \Lambda_n(E)+Z_n(E,\varphi) -\omega(E) F(X(\varphi,E),Y(\varphi,E)) \partial_\varphi X(\varphi,E) \Big)\, d\varphi.
\end{gather*}
Since $H_0(X(\varphi,E),Y(\varphi,E))\equiv E$, then from  \eqref{H0} and \eqref{cond02} it follows that
\begin{gather*}
X(\varphi,E)= - \sqrt {2 E }\cos\varphi+\mathcal O(E), \quad Y(\varphi,E)= \sqrt{2 E} \sin \varphi+\mathcal O(E), \\
F_n(X(\varphi,E),Y(\varphi,E))=\lambda_n \sqrt{2 E} \sin\varphi + \mathcal O(E),\quad Z_n(E,\varphi)=\mathcal O(E^{3/2})
\end{gather*}
as $E\to 0$ for all $\varphi\in \mathbb R$.
Thus
\begin{eqnarray*}
\Lambda_n(E)&=&\omega(E)\Big\langle \frac{1}{\omega(E)} (\sqrt{2 E} \sin\varphi + \mathcal O(E))( \lambda_n \sqrt{2 E} \sin\varphi + \mathcal O(E)) -Z_n(E,\varphi)\Big\rangle \\
& =& \lambda_n E +\mathcal O(E^2), \quad E\to 0.
\end{eqnarray*}
This implies that system \eqref{FulSys} satisfies the conditions of Lemma~\ref{Lem1} and the stability of the equilibrium is determined by the sign of $\lambda_n$.
\end{proof}

{\bf Example 1}. The system
\begin{gather}
\label{ex1}
\frac{dx}{dt}=y, \quad \frac{dy}{dt}=-\sin x +t^{-\frac l q} \kappa_l y \sin x + t^{-\frac nq} \lambda_n y,\quad t\geq 1,
\end{gather}
where $ \kappa_l$, $\lambda_n={\hbox{\rm const}}$, $0<l<n$, has the form of \eqref{FulSys} with  $H_0(x,y)=1-\cos x+ y^2/2$, $F_l(x,y)\equiv  \kappa_l  y \sin x$, $F_n(x,y)\equiv \lambda_n y$, $H_i(x,y)\equiv 0$ for $i\neq 0$ and $F_j(x,y)\equiv 0$ for $j\not\in\{n,m\}$.  It can easily be checked that system \eqref{ex1} satisfies the conditions of Theorem~\ref{Th1} with $h=\infty$. Therefore, the stability of the equilibrium $(0,0)$ is determined by the parameter $\lambda_n$ and the ratio $n/q$ (see Fig.~\ref{Fig1}).

\begin{figure}
\centering
\subfigure[$\displaystyle \frac n q=\frac 12$, $\lambda_n=-1$]{\includegraphics[width=0.3\linewidth]{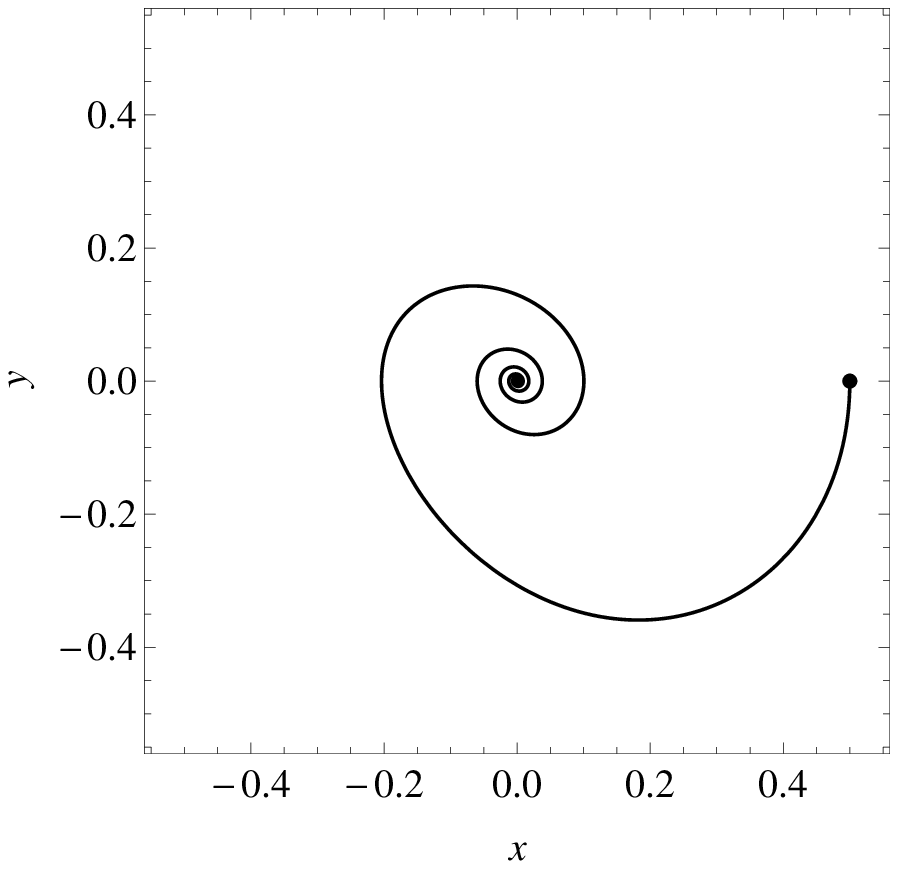}} \hspace{3ex}
\subfigure[$\displaystyle \frac n q=1$, $\lambda_n=-1$]{\includegraphics[width=0.3\linewidth]{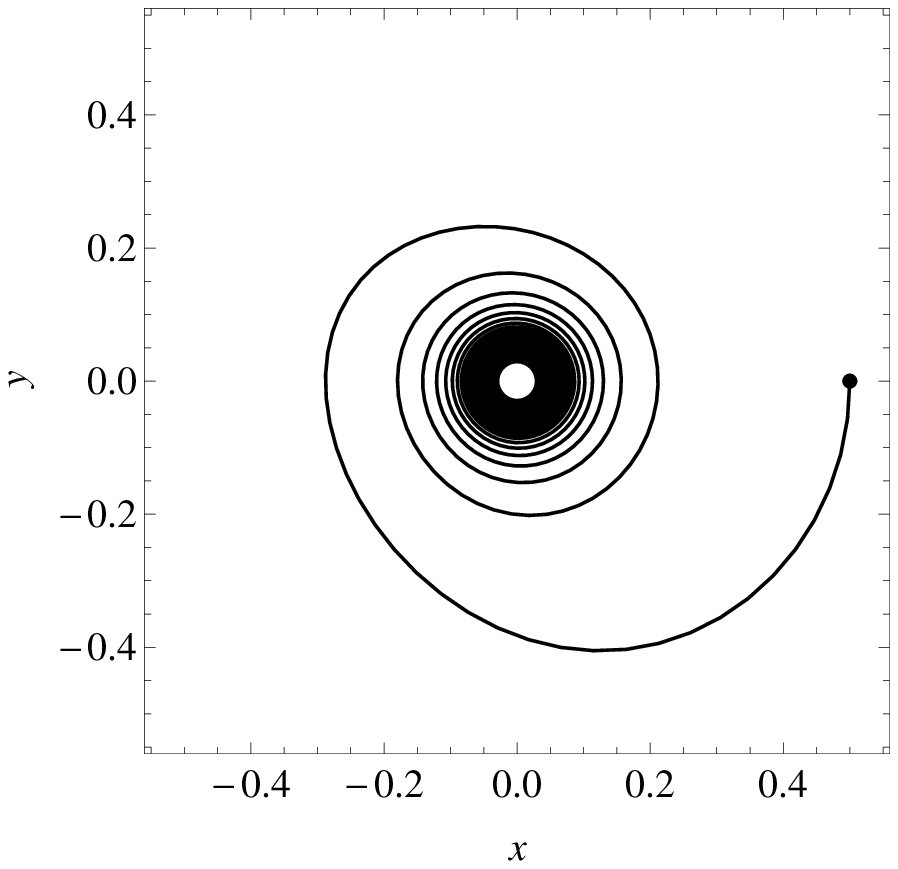}} \hspace{3ex}
\subfigure[$\displaystyle \frac n q= \frac 3 2$, $\lambda_n=-1$]{\includegraphics[width=0.3\linewidth]{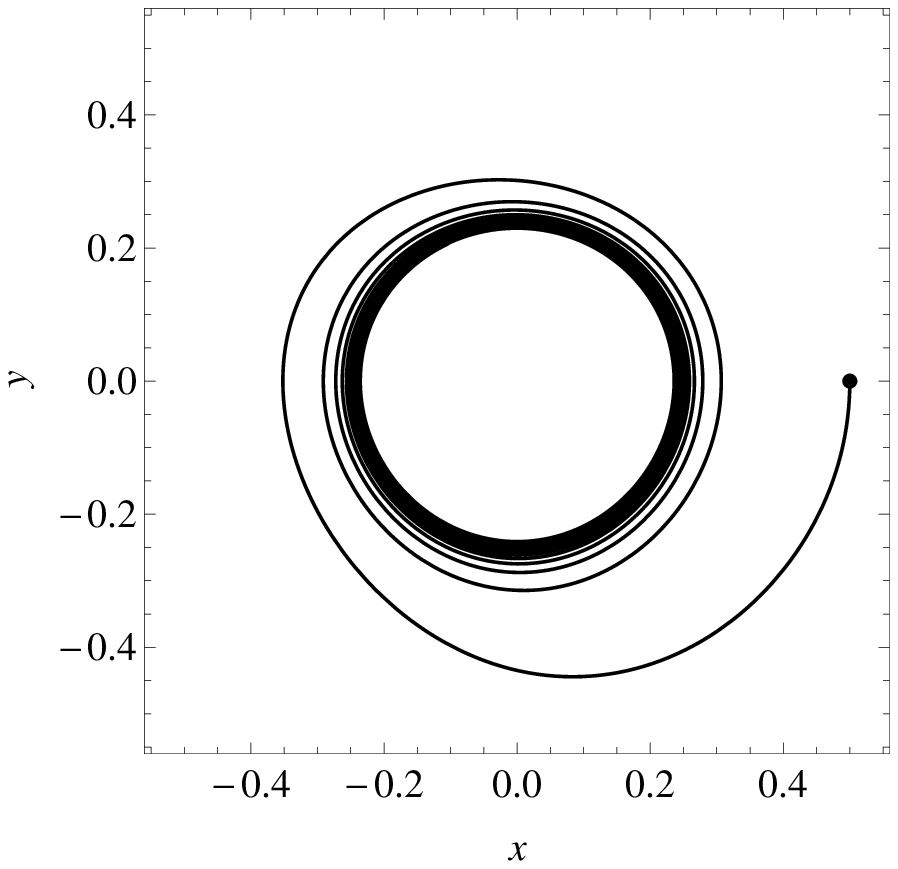}} \\
\centering
\subfigure[$\displaystyle \frac n q= \frac 12$, $\lambda_n=0.3$]{\includegraphics[width=0.3\linewidth]{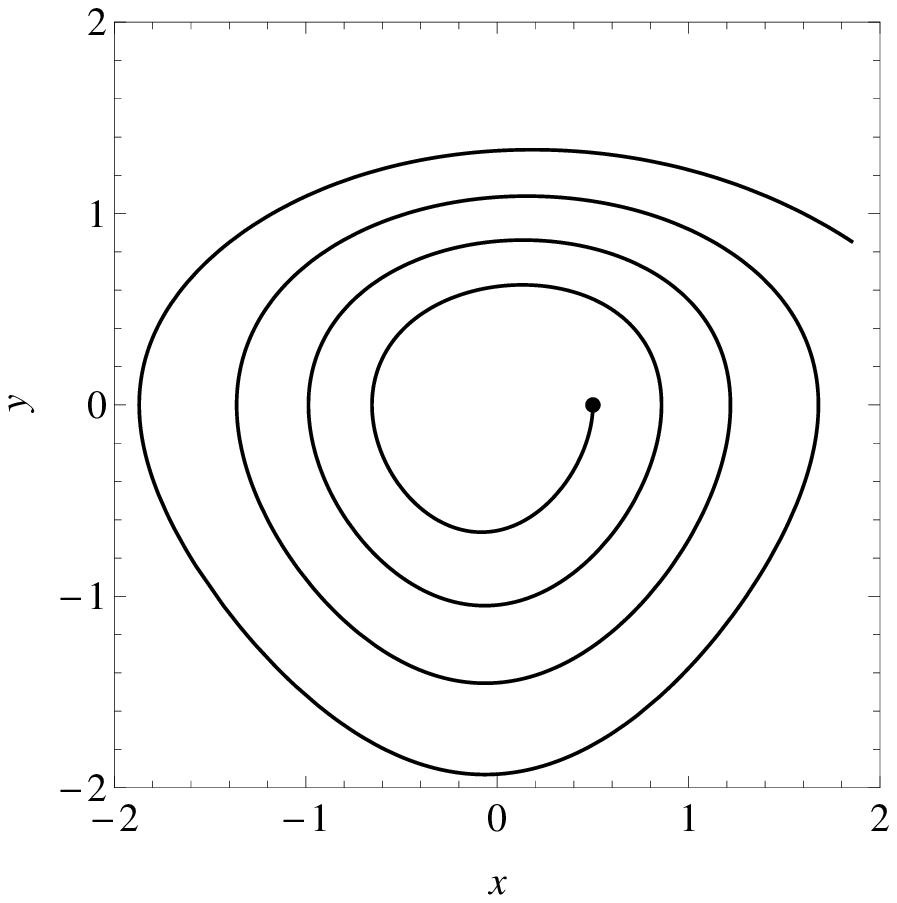}} \hspace{3ex}
\subfigure[$\displaystyle \frac n q=1$, $\lambda_n=0.3$]{\includegraphics[width=0.3\linewidth]{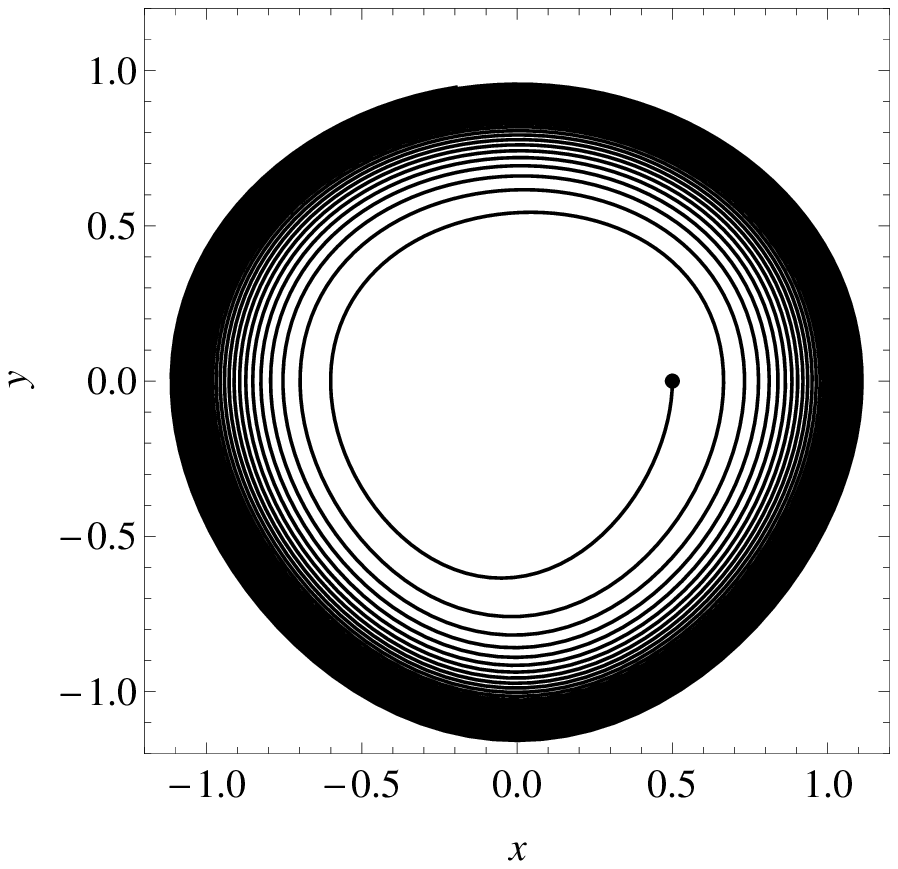}} \hspace{3ex}
\subfigure[$\displaystyle \frac n q=\frac 3 2$, $\lambda_n=0.3$]{\includegraphics[width=0.3\linewidth]{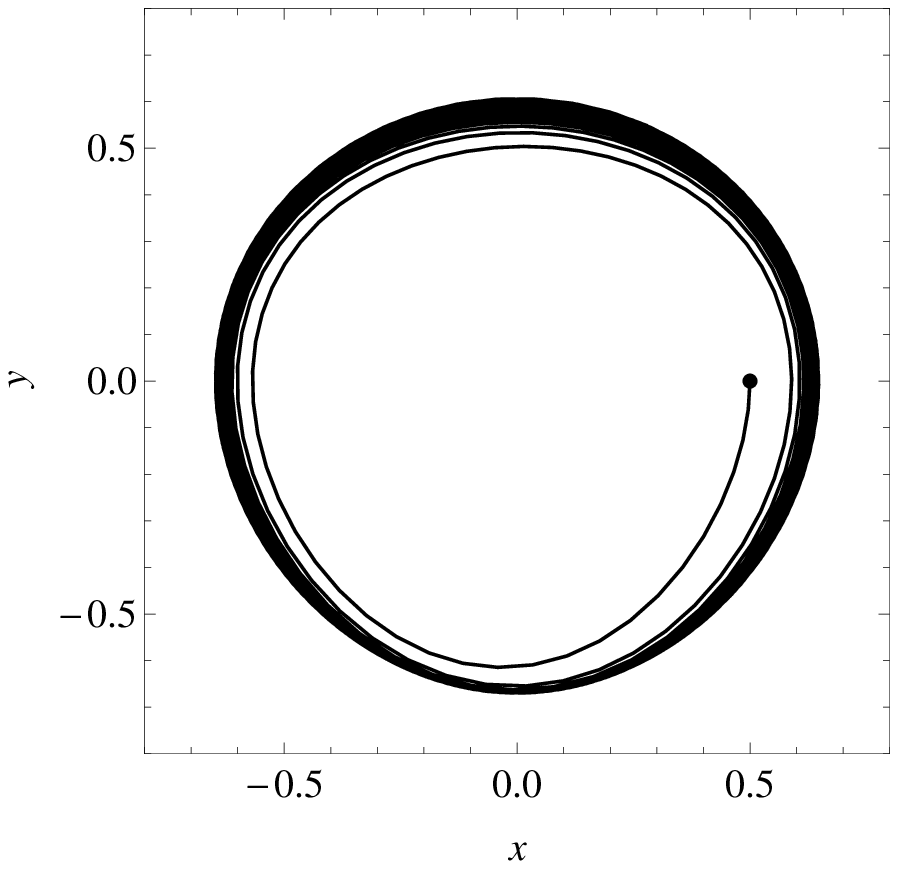}}
\caption{\small The evolution of $(x(t),y(t))$ for solutions of \eqref{ex1} with $q=4$, $l=1$, $\kappa_l=1$. The black points correspond to initial data $(0.5, 0)$. } \label{Fig1}
\end{figure}

\begin{Th}
\label{Th2}
Let $l, h, n, m$ be positive integers such that $l+h\geq n$, $n>m\geq l$ and the coefficients of the perturbations \eqref{HF} satisfy \eqref{cond00}, \eqref{cond02} and
\begin{gather*}
\oint\limits_{H_0(x,y)=E} F_{k}(x,y) \partial_y H_0(x,y)\, dl =0,\quad \forall\, E\in (0,E_0), \quad  l\leq  k<m, \\
F_m(x,y)= y (\alpha_m x^2 + \beta_m y^2)+\mathcal O(r^4), \quad r\to 0,
\end{gather*}
where $\alpha_m,\beta_m={\hbox{\rm const}}$. The equilibrium of system \eqref{FulSys} is
\begin{itemize}
  \item  stable if $\lambda_n<0$ and $\alpha_m+3\beta_m<0$;
  \item  unstable if $\lambda_n>0$ and $\alpha_m+3\beta_m>0$.
\end{itemize}
\end{Th}
\begin{proof}
The proof is similar to the proof of Theorem~\ref{Th1}. In this case, we show that  Lemma~\ref{Lem02} is applicable with $s=2$ and $\gamma_{m,s}=(\alpha_m+3\beta_m)/2$. Note that the change of the variables based on \eqref{VFm-l} leads to equation \eqref{VEq} with $\Lambda_k(E)\equiv 0$ for $ k< m$,
\begin{gather*}
\Lambda_i(E)\equiv \omega(E)\langle F_i \partial_\varphi X \rangle =\mathcal O(E^2), \quad E\to 0, \quad m\leq i\leq n-1,\\
\Lambda_m(E) =  4 E^2 \Big\langle \sin^2\varphi (\alpha_m \cos^2\varphi +\beta_m \sin^2\varphi)\Big\rangle +\mathcal O(E^3)=\gamma_{m,2}E^2 +\mathcal O(E^3).
\end{gather*}
The functions $\{\Lambda_i(E)\}_{i=m}^{n-1}$ are used in the calculating of $Z_j(E,\varphi)$, $j\geq m+h\geq n$. If $j=m+h=n$, then $Z_n(E,\varphi)=Z_n^1(E,\varphi)+Z_n^2(E,\varphi)+Z_n^3(E,\varphi)$, where $Z_n^{1}(E,\varphi)$ and $Z_n^{2}(E,\varphi)$ are defined by formulas \eqref{zn1} and \eqref{zn2},  $Z_n^3(E,\varphi)= \partial_E\Lambda_m(E) v_h(E,\varphi)=\mathcal O(E^2)$ as $E\to 0$. Therefore, $\Lambda_n(E)=\lambda_n E+\mathcal O(E^2)$ as $E\to 0$.
This means that system \eqref{FulSys} satisfies the conditions of Lemma~\ref{Lem02}.
\end{proof}

\begin{Cor}\label{CorTh2}
Under the conditions of Theorem~\ref{Th2}, Lemma~\ref{Lem2} is applicable to system \eqref{FulSys} with $s=2$ and $\gamma_{m,s}=(\alpha_m+3\beta_m)/2$.
\end{Cor}

{\bf Example 2.}
The system
\begin{gather}\label{ex2}
\frac{dx}{dt}=y, \quad \frac{dy}{dt}=-\sin x +t^{-\frac mq} \alpha_m x^2 y +t^{-\frac nq}\lambda_n y, \quad t\geq 1,
\end{gather}
where $\alpha_m, \lambda_n={\hbox{\rm const}}$, $0<m<n$, demonstrates the inefficiency of linear stability analysis for non-autonomous systems of the form \eqref{FulSys}.
The matrix of the linearized system $x'=y$, $y'=-x+\lambda_n t^{-n/q} y$ has the following eigenvalues:
\begin{gather*}
\mu_\pm(t) = \frac{1}{2}\Big(\lambda_n t^{-\frac nq}\pm i \sqrt{4-\lambda_n^2 t^{-\frac{2n}{q}}}\Big).
\end{gather*}
Let $\lambda_n>0$, then $\Re \mu_+(t)>0$ for all $t\geq 1$. From Theorem~\ref{Th1} it follows that the equilibrium $(0,0)$ is unstable in the linearized system. On the other hand, system \eqref{ex2} satisfies the conditions of Theorem~\ref{Th2} with $l=m$ and $F_m(x,y)=\alpha_m x^2 y$. From Corollary~\ref{CorTh2} it follows that the equilibrium is stable if $\lambda_n>0$ and $\alpha_m<0$ (see Fig.~\ref{Fig2}).
\begin{figure}
\centering
\subfigure[$\alpha_m=0$, $\lambda_n=0.4$ ]{\includegraphics[width=0.4\linewidth]{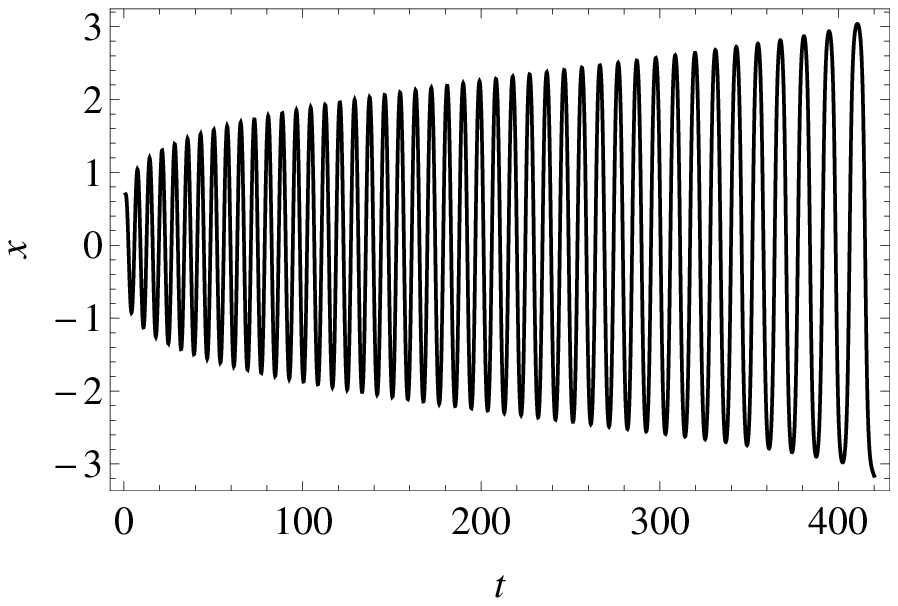}}
\hspace{4ex}
\subfigure[$\alpha_m=-2$, $\lambda_n=0.4$]{\includegraphics[width=0.4\linewidth]{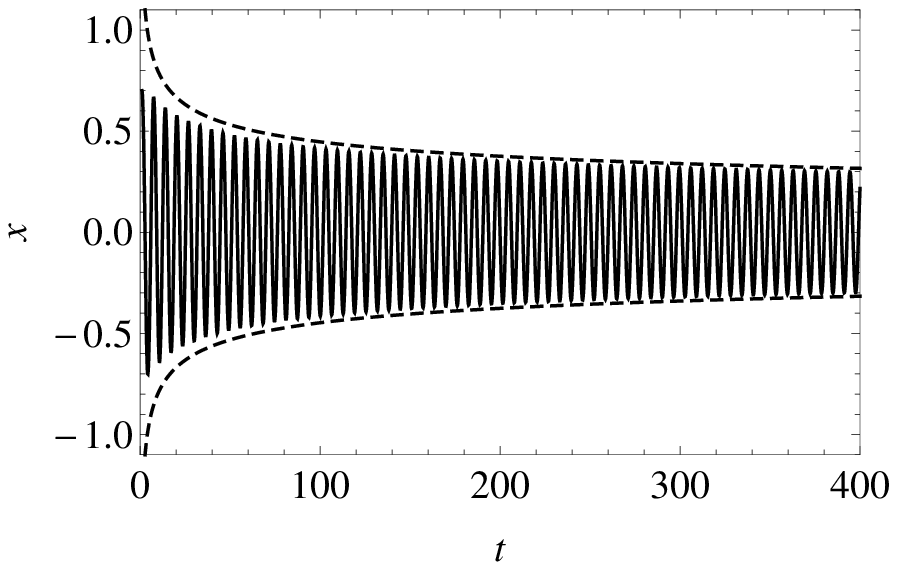}}
\caption{\small The evolution of $x(t)$ for solutions of \eqref{ex2} with $q=2$, $m=1$ and $n=2$. The dashed lines correspond to $\pm t^{\nu/2}$, $  \nu=\ (n-m)/q=0.5$.} \label{Fig2}
\end{figure}

{\bf Example 3.}
The system
\begin{gather}\label{ex3}
\frac{dx}{dt}=y, \quad \frac{dy}{dt}=-\sin x +t^{-\frac mq} \delta_m x^4 y + t^{-\frac nq} \alpha_n x^2 y, \quad t\geq 1
\end{gather}
has the form of \eqref{FulSys} with $H_0(x,y)=1-\cos x+y^2/2$, $F_m(x,y)=\delta_m x^4 y$, $F_n(x,y)=\alpha_n x^2 y$, $H_i(x,y)\equiv 0$ for $i\neq 0$, and $F_j(x,y)\equiv 0$ for $j\not\in\{ m,n\}$. It can easily be checked that
\begin{align*}
&Z_{k}(E,\varphi)\equiv 0, \quad \Lambda_{k}(E)\equiv 0, \quad k<m,\\
&Z_m(E,\varphi)\equiv 0,  \quad \Lambda_m(E)=\oint\limits_{H_0(x,y)=E} F_{m}(x,y) \partial_y H_0(x,y)\, dl  =  \gamma_{m,3} E^3 (1+\mathcal O(E)), \\
&Z_{j}(E,\varphi)=\mathcal O(E^3), \quad \Lambda_{j}(E)=\mathcal O(E^3), \quad m\leq j< n,\\
&Z_n(E,\varphi)=\mathcal O(E^3), \quad \Lambda_n(E)=\oint\limits_{H_0(x,y)=E} F_{n}(x,y) \partial_y H_0(x,y)\, dl - \langle Z_n(E,\varphi)\rangle = \gamma_{n,2} E^2 (1+\mathcal O(E)),
\end{align*}
as $E\to 0$, where $\gamma_{m,3}=\delta_m/2$ and $\gamma_{n,2}=\alpha_n/2$.
Therefore, system \eqref{ex3} satisfy the condition of Lemma~\ref{Lem22} with $s=3$ and $d=2$.
This implies that the equilibrium $(0,0)$ is stable if $\gamma_{m,3}<0$ and $\gamma_{n,2}<0$ (see Fig.~\ref{Fig3}).

\begin{figure}
\centering
\subfigure[$\delta_m=-0.3$, $\alpha_n=-0.3$]{\includegraphics[width=0.4\linewidth]{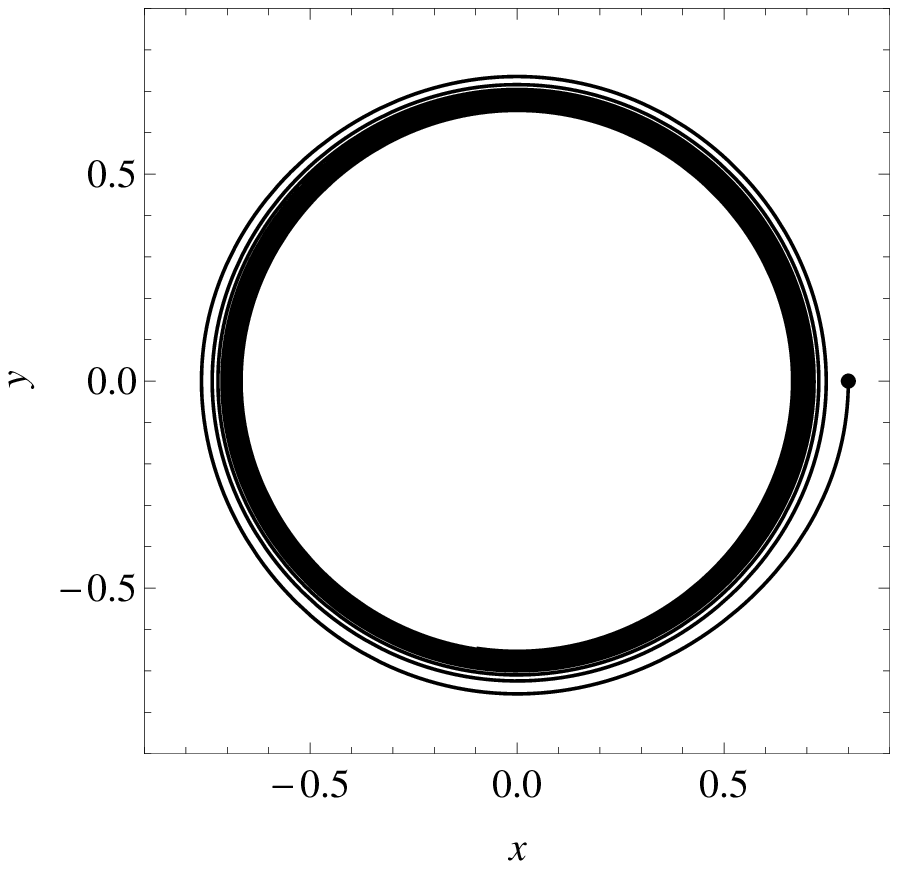}}
\hspace{4ex}
\subfigure[$\delta_m=0.3$, $\alpha_n=0.3$]{\includegraphics[width=0.4\linewidth]{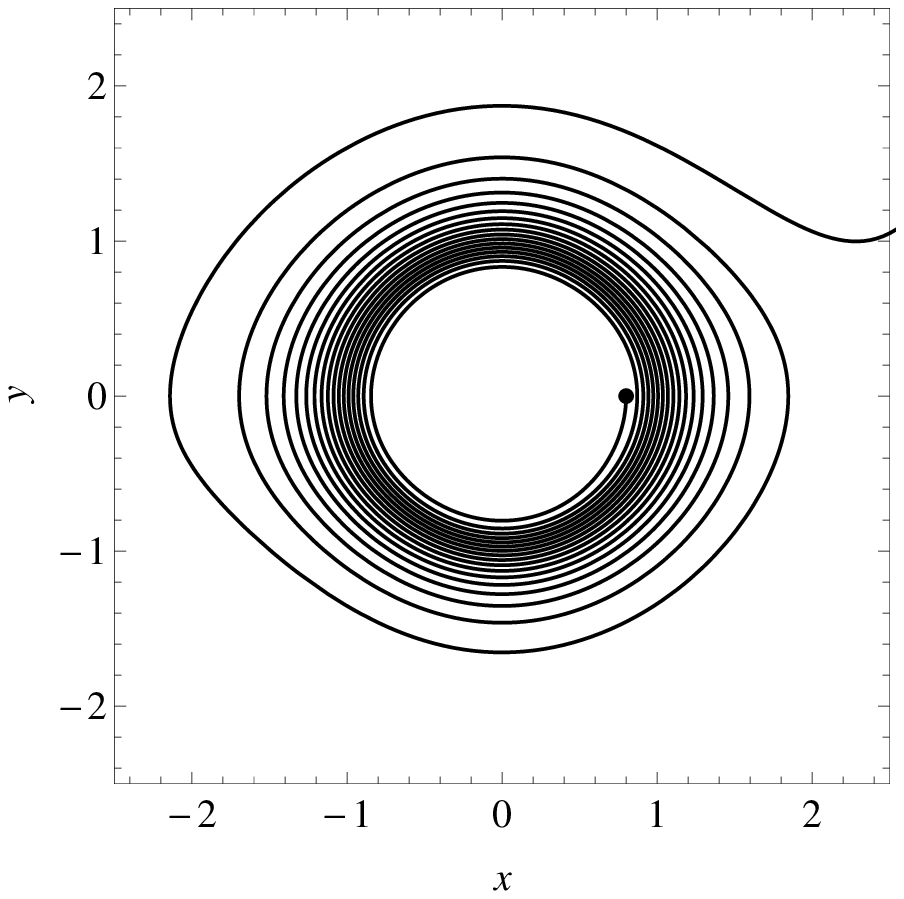}}
\caption{\small The evolution of $(x(t),y(t))$ for solutions of \eqref{ex3} with $q=2$, $m=1$ and $n=2$. The black points correspond to initial data.} \label{Fig3}
\end{figure}

\begin{Th}
\label{Th3}
Let $n$ be a positive integer such that $1\leq n\leq q$ and the perturbations \eqref{HF} satisfy the following conditions:
\begin{gather}\label{cond3}
    F_i(x,y)\equiv 0, \quad 1\leq i<n; \quad F_{n}(x,y)\equiv  \partial_y H_0(x,y) \big(\lambda_n-\mu_n H_0(x,y)\big)+\hat F_n(x,y),
\end{gather}
where $\lambda_n,\mu_n={\hbox{\rm const}}\neq 0$ and $\hat F_n(x,y)$ satisfies \eqref{cond01}. Then the equilibrium $(0,0)$ of system \eqref{FulSys} is unstable if $\lambda_n>0$ and is asymptotically stable if $\lambda_n<0$. Moreover, if $\lambda_n>0$, $\mu_n>0$ and $|\lambda_n/\mu_n|< E_0$, there exists stable limit cycle $H_0(x,y)=\lambda_n/\mu_n$.
\end{Th}
\begin{proof}
The proof is based on the application of Lemma~\ref{Lem1} and Lemma~\ref{Lem3}. From \eqref{cond3} it follows that there exists the transformation $(x,y)\mapsto(E,\varphi)\mapsto (v,\varphi)$ with
\begin{gather*}
V(E,\varphi,t)=E+\sum_{k=1}^n t^{-\frac k q}v_k(E,\varphi),
\end{gather*}
which reduces  system \eqref{FulSys} to the following form:
\begin{gather*}
\frac{dv}{dt}=t^{-\frac nq}\Lambda_n(v)+R_{n+1}(v,\varphi,t), \quad \frac{d\varphi}{dt}=\omega(v)+G(v,\varphi,t),\\
R_{n+1}(v,\varphi,t)=\mathcal O(t^{-\frac{n+1}{q}}), \quad G(v,\varphi,t)=\mathcal O(t^{-\frac 1q}), \quad t\to\infty, \\
\Lambda_n(E)\equiv \omega(E)\langle F_n(X(\varphi,E),Y(\varphi,E)) \partial_\varphi X(\varphi,E) \rangle\equiv  (\lambda_n-\mu_n E) \langle \big(\partial_y H_0\big)^2\rangle.
\end{gather*}
Note that  $\Lambda_n(E)=\lambda_n E +\mathcal O(E^2)$ as $E\to 0$. Hence, the equilibrium is asymptotically stable if $\lambda_n<0$ and the equilibrium is unstable if $\lambda_n>0$.

If $\lambda_n>0$ and $\mu_n>0$, the equation $\Lambda_n(E)=0$ has the root $E_c=\lambda_n/\mu_n$ such that $\Lambda_n'(E_c)<0$. In this case, from  Lemma~\ref{Lem3} it follows that there is a stable limit cycle $H_0(x,y)=E_c$.
\end{proof}

\begin{Rem}
If $\lambda_n<0$, $\gamma_n<0$ and $|\lambda_n/\mu_n|< E_0$, the fixed point $(0,0)$ is asymptotically stable and the limit cycle $H_0(x,y)=E_c$ with $E_c=\lambda_n/\mu_n$ is instable. If $E_c$ is the minimal nonzero root of $\Lambda_n(E)$, then $\{(x,y): 0 \leq H_0(x,y)<E_c\}$ is the domain of attraction of the fixed point $(0,0)$ .
\end{Rem}

{\bf Example 4.}
The system
\begin{gather}
\label{ex4}
\frac{dx}{dt}=y, \quad \frac{dy}{dt}=-x +t^{-\frac nq} y\Big(\lambda_n+\kappa_n x-\mu_n \frac{(x^2 + y^2)}{2}\Big), \quad t\geq 1
\end{gather}
with $\lambda_n,\mu_n,\kappa_n={\hbox{\rm const}}$ satisfies the conditions of Theorem~\ref{Th3} with $H_0(x,y)=(x^2+y^2)/2$ and $\hat F_n(x,y)=\kappa_n xy$. Therefore, if $\lambda_n>0$ and $\mu_n>0$, the system has the stable limit cycle $H_0(x,y)=\lambda_n/\mu_n$ (see Fig.~\ref{Fig4}).
\begin{figure}
\centering
\subfigure[$\lambda_n=\mu_n=0.5$]{\includegraphics[width=0.4\linewidth]{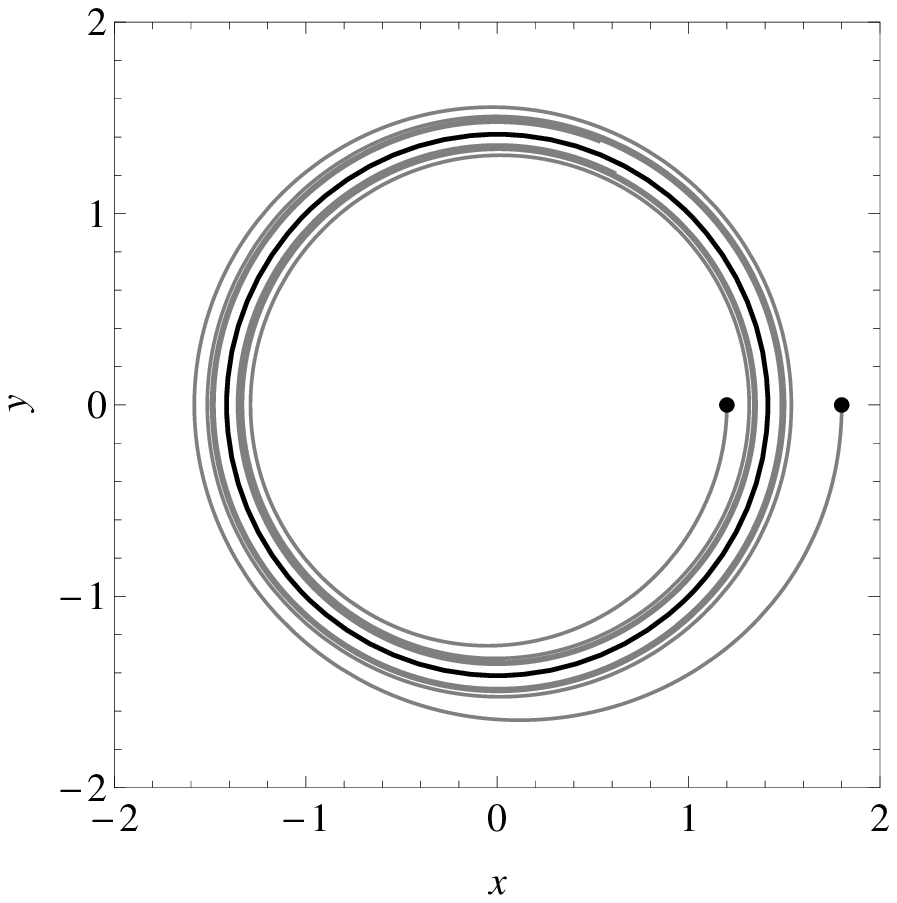}}
\hspace{4ex}
\subfigure[$\lambda_n=\mu_n=-0.5$]{\includegraphics[width=0.4\linewidth]{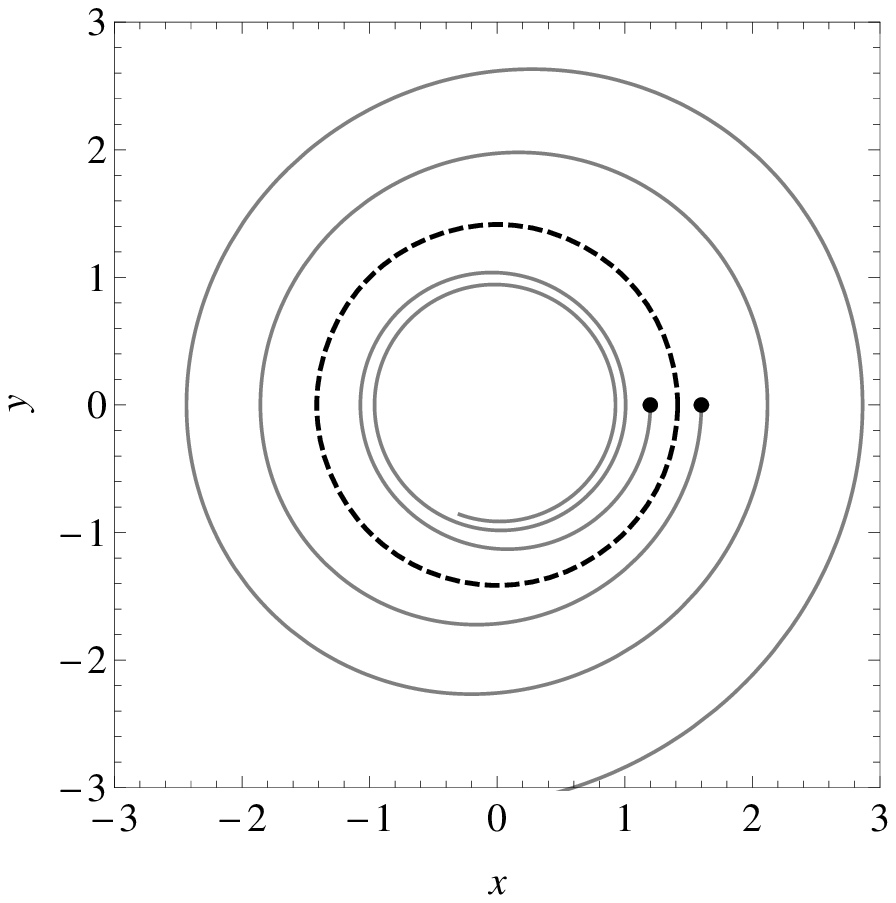}}
\caption{\small The evolution of $(x(t),y(t))$ for solutions of \eqref{ex4} with $n/q=1$ and $\kappa_n=0$ as $t\geq 1$. The black points correspond to initial data. The black solid (dashed) line corresponds to the stable (unstable) limit circle $H_0(x,y)=1$.} \label{Fig4}
\end{figure}

\section{Conclusion}
Thus, we have described possible bifurcations in asymptotically autonomous Hamiltonian systems in the plane. The important feature of non-autonomous systems is the inefficiency of the linear stability analysis: there are examples of nonlinear systems whose solutions behave completely differently than the solutions of corresponding linearized system.
Here, through a careful nonlinear analysis based on the Lyapunov function method we have shown that depending on the structure of decaying perturbations the equilibrium of the limiting system can preserve or lose stability. Note also that in this paper perturbations preserving the equilibrium of a Hamiltonian system have been considered. If the equilibrium disappears in the perturbed equations, instead of the equilibrium, a particular solution of a perturbed system should be considered.

\section*{Acknowledgments}
The research is supported by the Russian Science Foundation (Grant No. 20-11-19995).

}

\begin{thebibliography}{99}
\bibitem{GH83} J. Guckenheimer, P. Holmes, \textit{ Nonlinear oscillations, dynamical systems and bifurcations of vector fields}, Springer, New York, 1983.
\bibitem{GL94} P. A. Glendinning, \textit{ Stability, instability and chaos: an introduction to the theory of nonlinear differential equations}, Cambridge University Press, Cambridge, 1994.
\bibitem{Hans07} H. Han{\ss}mann, \textit{ Local and semi-local bifurcations in Hamiltonian systems - Results and examples}, Lecture Notes in Mathematics, 1893, Springer, Berlin, 2007.
\bibitem{Markus56} L. Markus, \textit{ Asymptotically autonomous differential systems}. In: S. Lefschetz (ed.), Contributions to the theory of nonlinear oscillations III, Ann. Math. Stud., vol. 36, pp. 17--29, Princeton University Press, Princeton, 1956.
\bibitem{WongBurton65} J. S. W. Wong, T. A. Burton, \textit{Some properties of solutions of $u''(t)+a(t)f(u)g(u')=0$. II}, Monatsh. Math., 69 (1965) 368--374.
\bibitem{Grimmer69}  R. C. Grimmer, \textit{Asymptotically almost periodic solutions of differential equations}, SIAM J. Appl. Math., 17  (1968) 109--115.
\bibitem{Theim92} H. R. Thieme, \textit{Convergence results and a Poincar\'{e}-Bendixson trichotomy for asymptotically autonomous differential equations}, J. Math. Biol., 30 (1992) 755--763.
\bibitem{Theim94} H. Thieme,  \textit{Asymptotically autonomous differential equations in the plane}, Rocky Mountain J. Math., 24 (1994) 351--380.
\bibitem{LRS02} J. A. Langa, J. C. Robinson, A. Su\'{a}rez, \textit{Stability, instability and bifurcation phenomena in nonautonomous differential equations}, Nonlinearity, 15 (2002) 887--903.
\bibitem{KS05} P. E. Kloeden, S. Siegmund, \textit{Bifurcations and continuous transitions of attractors in autonomous and nonautonomous systems}, Internat. J. Bifur. Chaos., 15 (2005) 743--762.
\bibitem{Rassmusen08} M. Rasmussen, \textit{Bifurcations of asymptotically autonomous differential equations}, Set-Valued Anal., 16  (2008)  821--849.
\bibitem{CP10} C. P\"{o}tzsche, \textit{Nonautonomous bifurcation of bounded solutions I: A Lyapunov-Schmidt approach}, Discrete Contin. Dynam. Systems  B, 14 (2010) 739--776.
\bibitem{IKNF} A. S. Fokas, A. R. Its, A. A. Kapaev, V. Yu. Novokshenov, \textit{ Painlev\'{e} transcendents. The Riemann-Hilbert approach}, Mathematical Surveys and Monographs, vol. 128, Amer. Math. Soc., Providence, 2006.
\bibitem{OSLK13} L. A. Kalyakin, O. A. Sultanov, \textit{Stability of autoresonance models}, Differ. Equat., 49 (2013) 267--281.
\bibitem{OS16} O. A. Sultanov, \textit{Stability of capture into parametric autoresonance}, Proc. Steklov Inst. Math., 295, suppl. 1 (2016) 156--167.
\bibitem{PRK02} A. Pikovsky, M. Rosenblum, J. Kurths. \textit{Synchronization: a universal concept in nonlinear sciences}, Cambridge University Press, Cambridge, 2001.
\bibitem{LK14} L. A. Kalyakin, \textit{Synchronization in a nonisochronous nonautonomous system}, Theoret. and Math. Phys., 181 (2014) 1339--1348.
\bibitem{Vazov} W. Wasîw, \textit{ Asymptotic expansions for ordinary differential equations}, John Wiley and Sons, Inc., New York, 1966.
\bibitem{OS18} O. Sultanov, \textit{Stability and asymptotic analysis of the autoresonant capture in oscillating systems with combined excitation}, SIAM J. Appl. Math., 78 (2018) 3103--3118.
\bibitem{OS19} O. Sultanov, \textit{Lyapunov functions and asymptotic analysis of a complex analogue of the second Painlev\'{e} equation}, J. Physics: Conf. Ser., 1205 (2019) 012056.
\bibitem{OS20} O. A. Sultanov, \textit{Bifurcations of autoresonant modes in oscillating systems with combined excitation}, Studies in Appl. Math., 144 (2020) 213--241.
\bibitem{Hapaev} M. M. Hapaev, \textit{ Averaging in stability theory: a study of resonance multi-frequency systems}, Kluwer Academic Publishers, Dordrecht, Boston, 1993.
\bibitem{AKN06} V. I. Arnold, V. V. Kozlov, A. I. Neishtadt, \textit{Mathematical aspects of classical and celestial mechanics}, Springer, Berlin, 2006.
\end{thebibliography}
\end{document}